\documentclass[12pt,fullpage,leqno]{article}

\usepackage{amsfonts}
\usepackage{amsmath}
\usepackage{amssymb}
\usepackage{epsfig}
\usepackage{graphicx}
\usepackage{amsthm}
\usepackage{latexsym}
\input epsf

\usepackage[T1]{fontenc}

\usepackage{color}
\definecolor{purple}{rgb}{0.65, 0, 1}
\definecolor{orange}{rgb}{1,.5,0}
\definecolor{brown}{rgb}{.9,.73,.26}

\def\R{\hbox{\bf R}}
\def\Z{\hbox{\bf Z}}

\def\N{\hbox{\bf N}}

\def\e{\epsilon}

\newcommand{\abs}[1]{\left\lvert #1\right\rvert}
\newcommand{\norm}[1]{\left\lVert #1\right\rVert}
\newcommand{\lp}{\left(}
\newcommand{\rp}{\right)}
\newcommand{\half}{\frac{1}{2}}

\newcommand{\ba}{\begin{eqnarray}}
\newcommand{\ea}{\end{eqnarray}}

\newtheorem{theo}{\bf Theorem}[section]
\newtheorem{lem}[theo]{\bf Lemma}
\newtheorem{pro}[theo]{\bf Proposition}
\newtheorem{cor}[theo]{\bf Corollary}
\newtheorem{defi}[theo]{\bf Definition}
\newtheorem{rem}[theo]{\bf Remark}

\newcommand{\be}{\begin{equation}}
\newcommand{\ee}{\end{equation}}

\renewcommand{\N}{{\mathbb N}}
\renewcommand{\R}{{\mathbb R}}

\renewcommand{\Z}{{\mathbb Z}}
\renewcommand{\e}{\varepsilon}

\newenvironment{Proofc}[1]{\smallskip\par\noindent\textsc{#1}\quad}%
  {\hfill$\Box$\bigskip\par}

\nonstopmode

\begin{document}

\title{\bf Slow motion of particle systems as a limit\\ of a reaction-diffusion
 equation with \\ half-Laplacian in dimension one}
 \author{M.d.M. Gonz\'alez,\footnote{Universitat Polit\`ecnica de Catalunya, ETSEIB-MA1, Av. Diagonal 647, 08028 Barcelona, Spain.}\\
R. Monneau
\footnote{Universit\'e Paris-Est, Cermics, Ecole des Ponts ParisTech,
   6 et 8 avenue Blaise Pascal, Cit\'e Descartes, Champs-sur-Marne,
   77455 Marne-la-Vall\'ee Cedex 2.}
}
\maketitle

\vspace{20pt}



 \centerline{\small{\bf{Abstract}}}
{\small{We consider a reaction-diffusion equation with a
    half-Laplacian. In the case where the solution is independent on time,
    the model reduces to the Peierls-Nabarro model describing dislocations
    as transition layers in a phase field setting.
We introduce a suitable rescaling of the evolution equation, using a small
parameter $\varepsilon$. As $\varepsilon$ goes to zero, we show that the limit dynamics is
characterized by a system of ODEs describing the motion of particles with
two-body interactions. The interaction forces are in $1/x$ and correspond
to the well-known interaction between dislocations.}}\hfill\break

 \noindent{\small{\bf{AMS Classification:}}} {\small{35Q99, 35B40, 35J25, 35D30, 35G25, 70F99.}}\hfill\break
 \noindent{\small{\bf{Keywords:}}} {\small{non-local Allen-Cahn equation,
       reaction-diffusion equation, singular limit, Peierls-Nabarro model,
       transition layer, particle systems.}}\hfill\break

\section{Introduction}

Let us set the one-dimensional half-Laplacian operator $L$ defined on
functions $w\in L^\infty (\R)\cap C^{1,\beta}_{loc}(\R)$ for some
$0<\beta<1$ by the L\'evy-Khintchine formula
$$Lw(x)=\frac{1}{\pi} \int_{\R}\frac{dz}{z^2}\left(w(x+z)-w(x)-zw'(x) 1_{\left\{|z|\le
      1\right\}}\right).$$
Given a suitable $1$-periodic potential $W$ which satisfies in particular $W'(k)=0$ for $k\in\Z$ ($W$ will be precised later
in assumption \textbf{(A)}), it is known (see Section \ref{subsection-layer}) that there exists a unique solution $\phi$ of
\begin{equation}\label{eq::0}
\left\{\begin{array}{l}
L\phi -W'(\phi)=0 \quad \mbox{on}\quad \R,\\
\displaystyle{\phi'>0 \quad \mbox{and}\quad \phi(-\infty)=0,\quad \phi(0)=\frac 1 2,\quad \phi(+\infty)=1.}
\end{array}\right.
\end{equation}
Model \eqref{eq::0} is known as the Peierls-Nabarro model for dislocations
in crystals where $W$ is called the stacking fault
energy. Here $\phi(x)$ is a phase transition representation of a dislocation
localized at the position $x=0$.
For a physical introduction to the model, see for instance \cite{HL}; for a recent reference, see \cite{WXM}; we also refer the reader to the paper of Nabarro \cite{N} which presents an historical tour on the model.

\medskip

We now consider a scalar function $v$ solution of the following evolution problem
associated to the Peierls-Nabarro model:
\begin{equation}\label{eq::901}
v_t=Lv-W'(v) + \sigma^\varepsilon(t,x) \quad \mbox{on}\quad
(0,+\infty)\times \R
\end{equation}
where $\sigma^\varepsilon$ is the exterior stress acting in the material.
This equation has been for instance considered in \cite{MBW} (see also \cite{Denoual} for a similar model).
We assume that this exterior stress is given for $\varepsilon>0$ small, by
$$\sigma^\varepsilon(t,x)=\varepsilon\sigma\left(\frac{t}{\varepsilon^2},\frac{x}{\varepsilon}\right)$$
where $\sigma$ is a given suitable function.
Then we can consider the following rescaling
$$v^\varepsilon(t,x)=v\left(\frac{t}{\varepsilon^2},\frac{x}{\varepsilon}\right).$$
We now look for scalar functions  $v^\varepsilon$ satisfying
\begin{equation}\label{eq::1}
\left\{\begin{array}{l}
\displaystyle{v^\varepsilon_t=\frac{1}{\varepsilon}\left\{Lv^\varepsilon-\frac{1}{\varepsilon}W'(v^\varepsilon) +
  \sigma(t,x)\right\} \quad \mbox{on}\quad  (0,+\infty)\times \R}\\
\\
v^\varepsilon(0,x)=v^\varepsilon_0(x)  \quad \mbox{for}\quad x\in\R.
\end{array}\right.
\end{equation}

When $v^\varepsilon$ has $N$ transition layers as $\varepsilon$ goes to zero,
we will see that we can associate a particle to each transition layer such that
the dynamics of
$v^\varepsilon$ is driven by the following  ODE system for
particles $(x_i(t))_{i=1,...,N}$:
\begin{equation}\label{eq::2}
\left\{\begin{array}{lcl}
\displaystyle{\frac{dx_i}{dt}}= \gamma \left(-\sigma(t,x_i) + \frac{1}{\pi}\sum_{j\not= i} \frac{1}{x_i-x_j}\right),\quad
\mbox{on}\quad (0,+\infty),\\
x_i(0)=x_i^0,
\end{array}\right|  \quad
\mbox{for}\quad i=1,...,N
\end{equation}
where
\begin{equation}\label{eq::2'}
\gamma = \left(\int_{\R}(\phi')^2\right)^{-1}.
\end{equation}

\subsection{Statement of the main result}

To state precisely our main result, we make the following assumptions:\\
\noindent {\bf(A)}\\
\noindent {\bf i) Assumption on the potential}\\
$$\left\{\begin{array}{l}
W\in C^{2,\beta}(\R) \quad \mbox{for some}\quad 0<\beta<1,\\
\\
W(v+1)=W(v) \quad \mbox{for any}\quad v\in\R,\\
\\
W=0 \quad \mbox{on}\quad\Z,\\
\\
W>0 \quad \mbox{on}\quad\R\backslash  \Z,\\
\\
W''(0)>0.
\end{array}\right.$$
\noindent {\bf ii) Assumption on  the exterior stress}
$$\left\{\begin{array}{l}
\sigma \in BUC\left([0,+\infty )\times \R\right) \quad \mbox{and for some}\quad
0<\beta<1 \quad \mbox{and}\quad K>0, \quad \mbox{we have}\\
\\
|\sigma_x|_{L^\infty ([0,+\infty )\times \R)}\le K, \quad |\sigma_t|_{L^\infty
  ([0,+\infty )\times \R)}\le K,\\
\\
|\sigma_x(t,x+h)-\sigma_x(t,x)|\le K |h|^\beta \quad \mbox{for all}\quad
x,h\in\R,\quad t\in [0,+\infty ).
\end{array}\right.$$
We recall that BUC is the space of bounded uniformly continuous functions.\\
\noindent {\bf iii) Assumption on the initial condition}
\begin{equation}\label{eq::900}
\left\{\begin{array}{l}
\mbox{For}\quad   x_1^0< x_2^0<
...<x^0_N, \quad \mbox{we set}\\
\\
\displaystyle{v_0^\varepsilon(x)=\frac{\varepsilon}{\alpha}\sigma(0,x) + \sum_{i=1}^N
\phi\left(\frac{x-x_i^0}{\varepsilon}\right) \quad \mbox{with}\quad
\alpha=W''(0)>0.}
\end{array}\right.
\end{equation}

Let us now give our main result (which has been announced in \cite{EHIM}):

\begin{theo}\label{th::1}{\bf (Convergence to the ODE system)}\\
Under assumption (A), there exists a unique solution $(x_i(t))_{i=1,...,N}$
of (\ref{eq::2}). Let
$$v^0(t,x)= \sum_{i=1}^N H(x-x_i(t))$$
where $H(x)=1_{[0,+\infty)}(x)$ is the Heaviside function.
Then for any $\varepsilon >0$, there exists a unique viscosity solution
$v^\varepsilon$ of (\ref{eq::1}). Moreover as $\varepsilon \to 0$,
$v^\varepsilon$ converges to $v^0$ in the following sense:
\begin{equation}\label{eq::c+}
\limsup_{(t',x')\to(t,x),\ \e\to 0} v^{\e}(t',x')\leq (v^0)^* (t,x)
\end{equation}
and
\begin{equation}\label{eq::c-}
\liminf_{(t',x')\to(t,x),\ \e\to 0} v^{\e}(t',x')\geq (v^0)_* (t,x)
\end{equation}
for $(t,x)\in [0,+\infty )\times \R$.
\end{theo}

\begin{rem}
We recall here that the semi-continuous envelopes of a given any function
$u$ are defined as follows
$$u^*(t,x)=\limsup_{(t',x')\to(t,x)} u(t',x')\quad\quad\mbox{and}\quad\quad u_*(t,x)=\liminf_{(t',x')\to(t,x)} u(t',x').$$
\end{rem}

The proof of Theorem \ref{th::1} is done by constructing suitable sub and
supersolutions essentially based on the following ansatz of the solution:
\begin{equation}\label{eq::an}
\tilde{v}^\varepsilon(t,x)=\frac{\varepsilon}{\alpha}\sigma(t,x)+\sum_{i=1}^N
\left\{ \phi\left(\frac{x-x_i(t)}{\varepsilon}\right)-\varepsilon
  \dot{x}_i(t)\  \psi\left(\frac{x-x_i(t)}{\varepsilon}\right) \right\}
\end{equation}
where $\psi$ is a suitable corrector.

\subsection{Brief review of the literature}

\bigskip

As we have mentioned, the diffusion equation \eqref{eq::1}
is the evolution equation for the Peierls-Nabarro model, which describes the dynamics of dislocation defects in crystal.
For general references, we refer the reader to \cite{Alvarez-Hoch-LeBouar-Monneau:dislocation-dynamics} (see also the paper  \cite{Toland:Peierls-Nabarro}).
The Peierls-Nabarro model can also be derived from Frenkel-Kontorova models (see \cite{FinoIM}).

In the time-independent setting, our model is related to a pioneering work of
Alberti-Boucchitt\'e-Seppecher.
They have shown  in \cite{Alberti-Bouchitte-Seppecher:singular-perturbations}  that if we consider the energy functional
for $\Omega\subset\mathbb R^n$ and a double well potential $W$ given by
$$E_\e[u]=\frac{1}{2}\int_{\Omega} \abs{\nabla u}^2+\frac{1}{\e}\int_{\partial\Omega}W(u)$$
then, in particular, the rescalings
$E_\e[u]/\abs{\log\e}$
Gamma-converge when $\e\to 0$ to $\frac{2}{\pi}\mathcal H^0(S_{u_0})$ if $n=2$ (thus $\partial\Omega$ is one dimensional) and the minimizers $u_\e$ converge to
$u_0$, where $u_0=0,1$ on $\partial\Omega$, $\Delta u_0=0$ in $\Omega$, and the singular set is
$$S_{u_0}=\left\{ z\in\partial\Omega \mbox{ : }u_0 \mbox{ jumps at }z \right\}.$$
A generalization to the case of elasticity equations with application to dislocations has been done by Garroni and M\"{u}ller
in \cite{Garroni-Muller:Gamma-limit-dislocations}.\\

For the time-dependent case, the (anisotropic) mean curvature motion has been obtained at the limit (in the framework of viscosity solutions) by Imbert and Souganidis \cite{IS}. Related to this result, let us also mention the work \cite{DaLio-Forcadel-Monneau:convergence}.

It is important to mention that a result similar to Theorem \ref{th::1} has been obtained for the non linear heat equation (after a suitable rescaling)
$$v_t =\varepsilon^2 v_{xx} - W'(v)$$
Here the interaction force between particles is also related to the behavior of the layer solution associated to the non linear heat equation, and hence is of exponential type.
For such results using an invariant manifold approach, we refer the reader to Carr-Pego \cite{Carr-Pego:metastable-patterns}, Fusco-Hale \cite{Fusco-Hale:slow-motion}, and Ei \cite{Ei} for generalization to systems of PDEs.
 For results using the energy approach, see Bronsard-Kohn \cite{Bronsard-Kohn:slowness}, Kalies, Van der Vorst, Wanner \cite{KVW}. Remark that our approach by sub and super solutions seems new (even in this context).
We also refer to the paper of Chen \cite{Chen}, for an interesting tour of the different regimes of the solution to the
non linear heat equation that appear for general initial data.

Let us mention that similar results have also been obtained for Cahn-Hilliard equations or their generalization to systems, called Cahn-Morral systems (see Grant \cite{Grant}).

We expect (even if we have no proof) Theorem \ref{th::1} to be true for any power of the Laplacian $(-\Delta)^\alpha$, $\alpha\in(0,1)$, once the layer solutions of \eqref{eq::0} are found. The existence of such layer solutions is the content of the work in preparation \cite{Cabre-Sire:layer-solutions}, that tries to generalize the work of Cabr\'e and Sol\`a-Morales  \cite{Cabre-SolaMorales:layer-solutions} for layer solutions of the half Laplacian. The extension problem (see \cite{Caffarelli-Silvestre}) that would be needed in this case is a degenerate elliptic equation with $A_2$ weight, that has been well understood in the classical reference \cite{Fabes-Kenig-Serapioni:local-regularity-degenerate}.

We should mention the result by Forcadel-Imbert-Monneau \cite{Forcadel-Imbert-Monneau:homogeneization}, where they homogenize in particular system (\ref{eq::2}) and more generally
non-local first order Hamilton-Jacobi equations describing the dynamics of dislocation lines. Similarly, homogenization results have been obtained in \cite{MP} for equation (\ref{eq::901}) for periodic stress $\sigma$.

\subsection{Organization of the paper}

In Section 2 we review of the notion and properties of viscosity solutions for this type of non-local equations. Section 3 contains all the preliminary results that will be required later (layer solutions, corrector) together with some motivation of our result and a study of the ODE system of particles \eqref{eq::2}. Section 4 is the proof of the main result (convergence is proven by finding suitable sub and supersolutions of the PDE).
Section 5 and Section 6 are almost independent of the rest of the paper: Section 5 is a necessary
review of known and further results (used in Section 6) for the half-Laplacian operator, while Section 6 contains the proof of the properties of the layer solution (Theorem \ref{th::5}) and of the existence of the corrector (Theorem \ref{th::6}).

\section{Viscosity solutions}

We recall the definition of a viscosity solution $v$ for equation (\ref{eq::1})
with $\varepsilon=1$ to simplify, namely for $T\in (0,+\infty )$:
\begin{equation}\label{eq::1ter}
\left\{
\begin{array}{l}
\displaystyle{v_t=Lv-W'(v) +
  \sigma(t,x) \quad \mbox{on}\quad  (0,T)\times \R}\\
v(0,x)=v_0(x)  \quad \mbox{for}\quad x\in\R.
\end{array}\right.
\end{equation}

\begin{defi}\label{defi::1}{\bf (Viscosity sub and supersolutions)}\\
An upper (respectively lower) semicontinuous function $v\in
L^\infty([0,T)\times \R)$ is said to be a
subsolution (resp. a supersolution) of (\ref{eq::1ter}) if and only if
$$v(0,\cdot)\le v_0^* \quad \mbox{on}\quad \R$$
$$\left(\mbox{resp.}\quad v(0,\cdot)\ge v_0^* \quad \mbox{on}\quad
\R\right)$$
and for any $C^2$ test function $\varphi$ satisfying
$$v= \varphi \quad \mbox{at the point}\quad (t_0,x_0)\in (0,T)\times \R$$
and
$$
v\le \varphi \quad \mbox{on a neighborhood of}\quad
\left\{t_0\right\}\times \R$$
$$\left(\mbox{resp.}\quad v\ge \varphi \quad \mbox{on a neighborhood of}\quad
\left\{t_0\right\}\times \R\right)$$
then we have
$$\varphi_t(t_0,x_0) \le L^1[\varphi(t_0,\cdot),v(t_0,\cdot)](x_0)
-W'(\varphi(t_0,x_0)) + \sigma(t_0,x_0)$$
$$\left(\mbox{resp.}\quad \varphi_t(t_0,x_0) \ge L^1[\varphi(t_0,\cdot),v(t_0,\cdot)](x_0)
-W'(\varphi(t_0,x_0)) + \sigma(t_0,x_0) \right)$$
where for any functions $\Phi(x), w(x)$ we set
$$L^1[\Phi,w](x_0)=\frac{1}{\pi}
\int_{[-1,1]}\frac{dz}{z^2}\left(\Phi(x_0+z)-\Phi(x_0)-z\Phi'(x_0)
\right)
 + \frac{1}{\pi}
\int_{\R\backslash [-1,1]}\frac{dz}{z^2}\left(w(x_0+z)-w(x_0)\right).$$
A function $v\in L^\infty([0,T)\times \R)$ is said to be a
viscosity solution of (\ref{eq::1ter}), if and only if $v^*$ is a subsolution and
$v_*$ is a supersolution.\\
Finally, a function $v\in L^\infty_{loc}([0,+\infty )\times \R)$ is said to
be a viscosity solution on $[0,+\infty)\times \R$, if and only if its is a
viscosity solution on $[0,T)\times \R$ for any $T>0$.

\end{defi}

\begin{rem} {\bf (Classical solutions are viscosity solutions)}\\
From this definition, we can easily check that every function
$v\in L^\infty([0,T)\times \R)\cap C^{1,\beta}_{loc}([0,T)\times \R)$ for
some $0<\beta<1$, is
indeed a viscosity solution if and only if it is a classical solution of
(\ref{eq::1ter}). This is mainly due to the fact that $v_t$ and $Lv(t,\cdot)$ are
then defined pointwise.
\end{rem}

Without entering in the details of the viscosity solutions (for which we refer
the interested reader the paper by Barles-Imbert \cite{Barles-Imbert:viscosity-solutions}), let us recall
the known results that we will use.

\begin{pro}\label{pro::v1}{\bf (Perron's method)}\\
Assume that there exists a supersolution $\underline{v}$ and a subsolution
$\overline{v}$ of equation (\ref{eq::1ter}) satisfying
$$\underline{v} \le \overline{v} \quad \mbox{on}\quad [0,T)\times \R.$$
Then there exists a viscosity solution $v$ of (\ref{eq::1ter}) satisfying
$$\underline{v} \le v\le \overline{v} \quad \mbox{on}\quad [0,T)\times \R.$$
\end{pro}

Let us also recall the comparison principle:

\begin{pro}\label{pro:4}{\bf (Comparison principle)}\\
Assume that $\underline{v}$ is a subsolution (resp. $\overline{v}$ is a supersolution) of equation
(\ref{eq::1ter}) on $[0,T)$ such that
$$\underline{v}(0,x)\le \overline{v}(0,x)\quad \mbox{for all}\quad  x\in \R.$$
Then
$$\underline{v}\le \overline{v}\quad \mbox{on}\quad  [0,T)\times \R.$$
\end{pro}

In particular, when the initial data $v_0$ is continuous, the comparison
principle implies the uniqueness and the continuity of the solution.

\section{Preliminary results}

In this section, we collect several results that will be used in the
section \ref{sect3} for the proof of our main result.

\subsection{Formal ansatz}

In the following, we try to explain formally why we have to expect an ansatz as in
\eqref{eq::an} for the solutions of \eqref{eq::1}.

\medskip

Namely, we are looking for an ansatz for the solutions of
\begin{equation}\label{eq::13}
\displaystyle{-\varepsilon
  v^\varepsilon_t=\frac{1}{\varepsilon}\left\{W'(v^\varepsilon)
  -\varepsilon Lv^\varepsilon - \varepsilon \sigma \right\}\quad \mbox{on}\quad  (0,+\infty)\times \R.}
\end{equation}
Assume to simplify that $\sigma$ is a constant here, and that there is only
one transition layer.\\

\noindent {\bf First try for the ansatz}\\
As a first try, in the case
where $\tilde{\sigma}$ is constant,
we could consider the Ansatz
$$\tilde{v}^\varepsilon(t,x)=\varepsilon \tilde{\sigma} +
\phi\left(\frac{x-x_1(t)}{\varepsilon}\right)$$
where $\phi$ is the
  transition layer solution of
\begin{equation}\label{eq::14}
L\phi-W'(\phi)=0.
\end{equation}
When we plug this expression in (\ref{eq::13}), we get
$$\begin{array}{lcl}
\dot{x}_1\phi' &\simeq &\displaystyle{\frac{1}{\varepsilon}\left\{W'(\varepsilon
  \tilde{\sigma} + \phi) - L\phi -\varepsilon \sigma\right\}}\\
\\
& \simeq &\displaystyle{\frac{1}{\varepsilon}\left\{W'(\varepsilon
  \tilde{\sigma} + \phi) - W'(\phi) -\varepsilon \sigma\right\}}\\
\\
& \simeq &  \displaystyle{W''(\phi) \tilde{\sigma}  -\sigma + O(\varepsilon)}\\
\end{array}$$
where we have used (\ref{eq::14}) for the second line, and a Taylor
  expansion to get the third line.

From this computation, we learn two things. First for $\phi \simeq 0$ we
  have to choose
\begin{equation}\label{eq::15}
W''(0) \tilde{\sigma}  =\sigma.
\end{equation}
The second information, is that it is impossible to satisfy the equation
with this ansatz, and we need to introduce a corrector.\\

\noindent {\bf Second try for the ansatz}\\
Let us take
$$\tilde{v}^\varepsilon(t,x)=\varepsilon \tilde{\sigma} +
\phi\left(\frac{x-x_1(t)}{\varepsilon}\right) -\varepsilon c
\psi\left(\frac{x-x_1(t)}{\varepsilon}\right)$$
for $\tilde{\sigma}$ satisfying (\ref{eq::15}) and for some constant $c$ to fix later.
When we plug this expression in (\ref{eq::13}), we get
$$\begin{array}{lcl}
\dot{x}_1\phi' +h.o.t. &\simeq &\displaystyle{\frac{1}{\varepsilon}\left\{W'(\varepsilon
  \tilde{\sigma} + \phi -\varepsilon c\psi) - L\phi +c\varepsilon L\psi -\varepsilon \sigma\right\}}\\
\\
& \simeq &\displaystyle{\frac{1}{\varepsilon}\left\{W'(\varepsilon
  \tilde{\sigma} + \phi -\varepsilon c\psi) - W'(\phi) +c\varepsilon L\psi
  -\varepsilon \sigma\right\}}\\
\\
& \simeq &  W''(\phi) (\tilde{\sigma}-c\psi)  -\sigma +c L\psi + h.o.t.
\end{array}$$
where we have used a Taylor expansion to get the last line. Finally using
(\ref{eq::15}), we see that $\psi$ has to solve
\begin{equation}\label{eq::16}
L\psi - W''(\phi)\psi = \frac{\dot{x}_1}{c} \phi' -
\frac{\tilde{\sigma}}{c}(W''(\phi)-W''(0)).
\end{equation}
We see that it is convenient to choose (in order to get later $\psi$ independent on $\dot{x}_1$)
$$\dot{x}_1=c.$$
Moreover, multiplying (\ref{eq::16}) by $\phi'$
we get formally by integration:
$$\begin{array}{lcl}
\displaystyle{\int_{\R} \phi'\left(L\psi - W''(\phi)\psi\right)} & = &\displaystyle{\int_{\R} (\phi')^2
-\frac{\tilde{\sigma}}{c} \left( \int_{\R} \frac{d}{dx}(W'(\phi(x)) -
  W''(0)\int_{\R} \phi'\right)}\\
\\
& = & \displaystyle{\int_{\R} (\phi')^2
+\frac{\tilde{\sigma}}{c} W''(0).}
\end{array}$$
On the other hand using the fact that $L$ is self-adjoint, we get formally
$$\begin{array}{lcl}
\displaystyle{\int_{\R} \phi'\left(L\psi - W''(\phi)\psi\right)} &= &\displaystyle{\int_{\R}
\psi\left(L\phi' - W''(\phi)\phi'\right)}\\
\\
&=& 0.
\end{array}$$
where the last equality comes from the fact that $\phi$ solves (\ref{eq::14}) so
then $\phi'$ solves
$$L\phi'-W''(\phi)\phi'=0.$$
Therefore, we get
$$c=-\gamma \sigma \quad \mbox{with}\quad \gamma=
\left(\int_{\R}(\phi')^2\right)^{-1}$$
and $\psi$ solves
$$L\psi - W''(\phi)\psi =  \phi' +\eta  (W''(\phi)-W''(0)) \quad
\mbox{with}\quad \eta= \frac{1}{\gamma W''(0)}.$$

\subsection{The layer solution and the corrector}\label{subsection-layer}

In this subsection, we state without proofs some results on the layer solution and the corrector.
These results will be proven in Section \ref{section-corrector}.

We recall that the layer solution satisfies (\ref{eq::0}), namely
\begin{equation}\label{eq::0bis}
\left\{\begin{array}{l}
L\phi -W'(\phi)=0 \quad \mbox{on}\quad \R,\\
\displaystyle{\phi'>0 \quad \mbox{and}\quad \phi(-\infty)=0,\quad \phi(0)=\frac12,\quad \phi(+\infty)=1}
\end{array}\right.
\end{equation}

Then we have
\begin{theo}\label{th::5}{\bf (Existence and properties of the layer solution)}\\
Under assumption Ai), there exists a unique solution $\phi$ of
(\ref{eq::0bis}).
Moreover $\phi \in C^{1,\beta}_{loc}(\R)$ for some $0<\beta <1$, and there exists a constant $C>0$ such that
\begin{equation}\label{eq::5}
0<\phi'(x)\le \frac{C}{1+x^2}  \quad \mbox{for all}\quad x\in\R
\end{equation}
and
\begin{equation}\label{eq::6}
\left|\phi(x)-H(x)+\frac{1}{\alpha \pi x}\right| \quad \le \quad \frac{C}{x^2}   \quad \mbox{for all}\quad |x|\ge 1
\end{equation}
where $H$ is the Heaviside function and $\alpha=W''(0)$.
\end{theo}

\medskip

We are also looking for a corrector $\psi$, solution of the following
equation
\begin{equation}\label{eq::7}
\left\{\begin{array}{l}
L\psi -W''(\phi)\psi=\phi'+ \eta (W''(\phi)-W''(0)) \quad \mbox{on}\quad \R,\\
\psi(-\infty)=0 = \psi(+\infty)
\end{array}\right.
\end{equation}
with
\begin{equation}\label{eq::7'}
\eta = \frac{\int_{\R}(\phi')^2}{W''(0)}.
\end{equation}

Then we have
\begin{theo}\label{th::6}{\bf (Existence and properties of the corrector)}\\
Under assumption Ai), there exists a unique solution $\psi \in
H^{\frac12}(\R)$ of (\ref{eq::7}).
Moreover $\psi \in C^{1,\beta}_{loc}(\R)\cap L^\infty (\R)$ for some $0<\beta< 1$ and
$$|\psi'|_{L^\infty(\R)} < +\infty. $$
\end{theo}


\subsection{The ODE system}

Recall that we consider a solution $(x_i(t))_{i=1,...,N}$ of (\ref{eq::2}),
namely
\begin{equation}\label{eq::2bis}
\left\{\begin{array}{lcl}
\displaystyle{\frac{dx_i}{dt}}= \gamma \left(-\sigma(x_i,t) + \frac{1}{\pi}\sum_{j\not= i} \frac{1}{x_i-x_j}\right),\quad
\mbox{on}\quad (0,+\infty),\\
x_i(0)=x_i^0,
\end{array}\right|  \quad
\mbox{for}\quad i=1,...,N
\end{equation}
where $\gamma >0$ is given in (\ref{eq::2'}).

We recall the following result

\begin{lem}\label{lem::1}{\bf (Lower bound on the distance between particles, \cite{Forcadel-Imbert-Monneau:homogeneization})}\\
Let $(x_i(t))_{i=1,...,N}$ be the solution of (\ref{eq::2bis}) on $[0,T]$ and let
$$d(t):=\min\{\abs{x_i(t)-x_j(t)},i\neq j\}$$
be the minimal distance between particles.
Then under assumptions Ai) and Aii), we have for all $t\in [0,T]$
$$d(t)\ge d(0) e^{-Ct} \quad \mbox{with}\quad C=\gamma K$$
where $K$ is the space Lipschitz constant of $\sigma$.
\end{lem}

This lemma prevents the crossing of particles in finite time. As a consequence, we easily get the following long time
existence result:
\begin{cor}\label{cor::1}{\bf (Long time
existence of a solution to the ODE system)}\\
Under assumption (A), there exists a unique solution $(x_i(t))_{i=1,...,N}$
of (\ref{eq::2bis}) on $[0,+\infty)$.
\end{cor}

\section{The convergence result: proof of Theorem \ref{th::1}}\label{sect3}

We want to build a supersolution with a parameter $\delta>0$ to fix
later. We define $(\overline{x}_i(t))_{i=1,...,N}$ as the solution of
$$\left\{\begin{array}{l}
\displaystyle{\frac{d\overline{x}_i}{dt}=
\gamma \left(-\delta-\sigma(\overline{x}_i,t) + \frac{1}{\pi}\sum_{j\not= i} \frac{1}{\overline{x}_i-\overline{x}_j}\right),\quad
\mbox{on}\quad (0,+\infty)},\\
\overline{x}_i(0)=x_i^0-\delta,
\end{array}\right|  \quad
\mbox{for}\quad i=1,...,N$$
with $\gamma$ given in  (\ref{eq::2'}).\\
We also define
$$\overline{v}^\varepsilon(t,x)=\varepsilon \tilde{\sigma}(t,x) +
\sum_{i=1}^N
\left\{\phi\left(\frac{x-\overline{x}_i(t)}{\varepsilon}\right)
  -\varepsilon\overline{c}_i(t)\psi\left(\frac{x-\overline{x}_i(t)}{\varepsilon}\right) \right\}$$
where
$$\left\{\begin{array}{l}
\overline{c}_i(t) = \dot{\overline{x}}_i(t)\\
\\
\displaystyle{\tilde{\sigma}=\frac{\delta+\sigma}{\alpha} \quad \mbox{with}\quad
\alpha=W''(0)>0}
\end{array}\right.$$
and $\phi$ is given in Theorem \ref{th::5} and $\psi$ in Theorem
\ref{th::6}.\\

Then we have
\begin{pro}\label{pro::2}{\bf (Supersolution at the initial time)}\\
Under assumption (A), there exists $\delta_0>0$ such that for all
$0<\delta\le \delta_0$, we have
\begin{equation}\label{eq::8}
\overline{v}^\varepsilon(0,x)\ge v_0^\varepsilon(x) \quad \mbox{for
  all}\quad x\in\R
\end{equation}
for $\varepsilon >0$ small enough.
\end{pro}

\noindent {\bf Proof of Proposition \ref{pro::2}}\\
First we choose $\delta>0$ small enough such that
$$x^0_i<\overline{x}_{i+1}(0)< x_{i+1}^0 \quad \mbox{for}\quad i=1,...,N-1.$$
Let $A>0$ be large enough such that
\begin{equation}\label{eq::10}
\frac{\delta}{\alpha}  \ge \left(\sum_{i=1}^N |\overline{c}_i(0)|\right)\
\sup_{\R\backslash [-A,A]} |\psi|.
\end{equation}
\noindent {\bf Case 1: $|x-\overline{x}_i(0)|\ge \varepsilon A$ for each
  $i=1,...,N$}\\
Using the fact that $\phi$ is increasing, we deduce that
$\displaystyle{\phi\left(\frac{x-\overline{x}_i(0)}{\varepsilon}\right) \ge
\phi\left(\frac{x-{x}_i^0}{\varepsilon}\right)}$
and then using (\ref{eq::10}), we get that
$$\overline{v}^\varepsilon(0,x)\ge v_0^\varepsilon(x).$$
\noindent {\bf Case 2: $|x-\overline{x}_{i_0}(0)| < \varepsilon A$ for some index
  $i_0\in \left\{1,...,N\right\}$}\\
Then
$$\phi\left(\frac{x-\overline{x}_{i_0}(0)}{\varepsilon}\right)\ge
\phi(-A)>0$$
while
$$\begin{array}{ll}
\displaystyle{\phi\left(\frac{x-{x}_{i_0}^0}{\varepsilon}\right)} & =
\displaystyle{\phi\left(\frac{x-\overline{x}_{i_0}(0)-\delta}{\varepsilon}\right)}\\
\\
& \displaystyle{\le \phi\left(A-\frac{\delta}{\varepsilon}\right) =
O\left(\frac{\varepsilon}{\delta}\right) \quad \mbox{as}\quad \varepsilon
\to 0.}
\end{array}$$
Therefore for $\varepsilon$ small enough, we have
$$\displaystyle{\phi\left(\frac{x-\overline{x}_{i_0}(0)}{\varepsilon}\right)
  \ge \phi\left(\frac{x-{x}_{i_0}^0}{\varepsilon}\right) +\sum_{i=1}^N
  \varepsilon \overline{c}_i(0) \psi\left(\frac{x-\overline{x}_{i}(0)}{\varepsilon}\right)}.$$
Using again the monotonicity of $\phi$, we conclude in case 2  that
$$\overline{v}^\varepsilon(0,x)\ge v_0^\varepsilon(x).$$
Finally case 1 and 2 prove that (\ref{eq::8}) holds for any $x\in\R$.
This ends the proof of the Proposition.
\qed\\

Let us define for $i=1,...,N$
$$\left\{\begin{array}{l}
\displaystyle{\psi_i :=\psi\left(\frac{x-\overline{x}_i(t)}{\varepsilon}\right)}\\
\\
\displaystyle{\tilde{\phi}_i :=\phi\left(\frac{x-\overline{x}_i(t)}{\varepsilon}\right) - H\left(\frac{x-\overline{x}_i(t)}{\varepsilon}\right)}
\end{array}\right.$$
where $H$ is the Heaviside function.

\begin{lem}\label{lem::2}{\bf (Computation using the ansatz)}\\
Let
$$\displaystyle{I^\varepsilon := \varepsilon \overline{v}^\varepsilon_t
+\frac{1}{\varepsilon}\left\{W'(\overline{v}^\varepsilon) -\varepsilon
  L\overline{v}^\varepsilon -\varepsilon \sigma\right\}}.$$
Then for any $i_0\in \left\{1,...,N\right\}$, we have
\begin{equation}\label{eq::12}
I^\varepsilon = e^\varepsilon_{i_0} + (\alpha \tilde{\sigma}-\sigma) +
O(\tilde{\phi}_{i_0}) \left\{\eta \overline{c}_{i_0} +\tilde{\sigma} + \sum_{i\not= i_0}
  \frac{\tilde{\phi}_i}{\varepsilon} \right\}
\end{equation}
where $\eta$ is defined in (\ref{eq::7'}) and the error
$e^\varepsilon_{i_0}$ is given by
$$e^\varepsilon_{i_0}= O(\varepsilon) + \sum_{i\not= i_0} O(\psi_i) +
\sum_{i\not= i_0} O(\tilde{\phi}_i) + \sum_{i\not= i_0}
  O\left(\frac{(\tilde{\phi}_i)^2}{\varepsilon}\right).$$
\end{lem}

\noindent {\bf Proof of Lemma \ref{lem::2}}\\
We have, using the equation for the corrector, that
$$\begin{array}{ll}
I^\varepsilon  =  & \displaystyle{\left(\sum_{i=1}^N
-\dot{\overline{x}}_i\phi'\left(\frac{x-\overline{x}_i(t)}{\varepsilon}\right)\right)
+  \left(\varepsilon\sum_{i=1}^N \left\{(\overline{c}_i)^2
  \psi'\left(\frac{x-\overline{x}_i(t)}{\varepsilon}\right) -
  \varepsilon\dot{\overline{c}}_i \psi_i\right\}\right)  + \varepsilon^2
\tilde{\sigma}_t}\\
\\
& + \displaystyle{\frac{1}{\varepsilon}}\left[\begin{array}{l}
-\varepsilon^2 L \tilde{\sigma} -\varepsilon \sigma\\
\\
\displaystyle{+ W'\left(\varepsilon \tilde{\sigma} + \sum_{i=1}^N \left\{\tilde{\phi}_i
    -\varepsilon \overline{c}_i \psi_i\right\}\right)}\\
\\
- \displaystyle{\sum_{i=1}^N \left\{W'(\tilde{\phi}_i) - \varepsilon \overline{c}_i
  \left[W''(\tilde{\phi}_i) \psi_i +
    \phi'\left(\frac{x-\overline{x}_i(t)}{\varepsilon}\right) + \eta
    \left(W''(\tilde{\phi}_i)-W''(0)\right)\right]\right\}}
\end{array}\right].
\end{array}$$
We first remark that all the terms in $\phi'$ vanish (because
$\overline{c}_i=\dot{x}_i$). Using assumption (A), we can bound
$\tilde{\sigma}_t, L\tilde{\sigma}$ and $\dot{\overline{c}}_i$. Moreover
using Theorem \ref{th::6}, we can also bound $\psi'$. Finally, collecting all the terms of order
$\varepsilon$ together, we get simply
$$I^\varepsilon  = O(\varepsilon) + \displaystyle{\frac{1}{\varepsilon}}\left[\begin{array}{l}
-\varepsilon \sigma\\
\\
\displaystyle{+ W'\left(\varepsilon \tilde{\sigma} + \sum_{i=1}^N \left\{\tilde{\phi}_i
    -\varepsilon \overline{c}_i \psi_i\right\}\right)}\\
\\
- \displaystyle{\sum_{i=1}^N \left\{W'(\tilde{\phi}_i) - \varepsilon \overline{c}_i
  \left[W''(\tilde{\phi}_i) \psi_i + \eta
    \left(W''(\tilde{\phi}_i)-W''(0)\right)\right]\right\}.}
\end{array}\right]$$
Now let us choose any index $i_0\in \left\{1,...,N\right\}$ and let us make
a Taylor expansion of $W'(\cdot)$ when the argument of this function is
close to $\tilde{\phi}_{i_0}$. We get
$$I^\varepsilon  = -\sigma +O(\varepsilon) + \displaystyle{\frac{1}{\varepsilon}}\left[\begin{array}{l}
\displaystyle{W'(\tilde{\phi}_{i_0}) + W''(\tilde{\phi}_{i_0})
  \left\{\varepsilon \tilde{\sigma} + \sum_{i\not= i_0} \tilde{\phi}_{i_0}
    + \left( -\varepsilon \overline{c}_{i_0} \psi_{i_0} +\sum_{i\not=i_0}
      -\varepsilon \overline{c}_i \psi_i\right)\right\}}\\
\\
\displaystyle{+ 0\left( \varepsilon^2 + \sum_{i\not= i_0} (\tilde{\phi}_i)^2\right)}\\
\\
-\displaystyle{\left\{W'(\tilde{\phi}_{i_0}) -
    \varepsilon\overline{c}_{i_0} \left[W''(\tilde{\phi}_{i_0}) \psi_{i_0}
      + \eta \left(W''(\tilde{\phi}_{i_0})-W''(0)\right)\right] \right\}}\\
\\
- \displaystyle{\sum_{i\not= i_0} \left\{W''(0)\tilde{\phi}_i +
    O((\tilde{\phi}_i)^2) - \varepsilon \overline{c}_i
  \left[W''(\tilde{\phi}_i) \psi_i + 0(\tilde{\phi}_i)\right]\right\}.}
\end{array}\right]$$
We remark that the terms in $W'(\tilde{\phi}_{i_0})$ vanish and the terms
in $\psi_{i_0}$ vanish. Now taking into account the definition of
$e^\varepsilon_{i_0}$, we see that we get
$$\displaystyle{I^\varepsilon  = -\sigma + e^\varepsilon_{i_0}
+(W''(\tilde{\phi}_{i_0})-W''(0))\left\{\tilde{\sigma} + \sum_{i\not= i_0}
  \frac{\tilde{\phi}_i}{\varepsilon}  + \eta \overline{c}_{i_0}\right\} +
W''(0) \tilde{\sigma}}$$
and then
$$\displaystyle{I^\varepsilon  =  e^\varepsilon_{i_0} +(\alpha
  \tilde{\sigma} -\sigma) + O(\tilde{\phi}_{i_0})\left\{\tilde{\sigma} + \sum_{i\not= i_0}
  \frac{\tilde{\phi}_i}{\varepsilon}  + \eta \overline{c}_{i_0}\right\}}.$$
This ends the proof of the lemma.
\qed\\

We are now ready to prove the following result:

\begin{pro}\label{pro::3}{\bf (Supersolution for positive time)}\\
Under assumption (A), there exists $\delta_0>0$ such that for all
$0<\delta\le \delta_0$ and given $T>0$, we have
\begin{equation}\label{eq::9}
\overline{v}^\varepsilon_t \ge \frac{1}{\varepsilon}
\left(L\overline{v}^\varepsilon - \frac{1}{\varepsilon}
W'(\overline{v}^\varepsilon) +\sigma(t,x) \right)   \quad \mbox{on}\quad
(0,T)\times \R
\end{equation}
for $\varepsilon >0$ small enough.
\end{pro}

\noindent {\bf Proof of Proposition \ref{pro::3}}\\
\noindent {\bf Case 1: $|x-\overline{x}_{i_0}(t)|\le \varepsilon^{\frac13}$
for some index $i_0\in \left\{1,...,N\right\}$}\\
Hence for $\varepsilon$ small enough, we also have
$$|x-\overline{x}_{i}(t)|\ge \varepsilon^{\frac13} \quad \mbox{for}\quad
i \in \left\{1,...,N\right\}\backslash \left\{i_0\right\}.$$
Recall also that by (\ref{eq::6}), we have
$$\left|\phi(y)-H(y)+\frac{1}{\alpha \pi y}\right| \quad \le \quad \frac{C}{y^2}   \quad \mbox{for all}\quad |y|\ge 1.$$
We deduce that
$$\left| \sum_{i\not= i_0} \left\{\frac{\tilde{\phi}_i}{\varepsilon} +
    \frac{1}{\alpha \pi (x-\overline{x}_i)}\right\} \right| \le (N-1) C
    \varepsilon^{\frac13} = O(\varepsilon^{\frac13}).$$
Therefore using (\ref{eq::12})
$$\begin{array}{lcl}
I^\varepsilon & \ge & \displaystyle{e^\varepsilon_{i_0} + O(\varepsilon^{\frac13}) +
 \alpha \tilde{\sigma}-\sigma}\\
\\
& & \displaystyle{+ 0(\tilde{\phi}_{i_0})\left\{ \frac{1}{\alpha}(\delta + \sigma(t,x))
   -\sum_{i\not=i_0} \frac{1}{\alpha\pi (\overline{x}_{i_0}-\overline{x}_i)}
 + \frac{1}{\alpha \gamma} \overline{c}_{i_0}\right\}}\\
\\
& \ge &  \displaystyle{e^\varepsilon_{i_0} +  O(\varepsilon^{\frac13}) + \delta +
 0(\tilde{\phi}_{i_0})\left\{
   \sigma(t,x)-\sigma(t,\overline{x}_{i_0})\right\}}\\
\\
& \ge & \delta +  O(\varepsilon^{\frac13}) + e^\varepsilon_{i_0}\\
\\
& \ge & \delta/2 \quad \mbox{for}\quad \varepsilon \quad \mbox{small enough}
\end{array}$$
where in the first inequality we have used the fact that $\alpha \gamma
\eta =1$. To get the last line, we have used the fact that $e^\varepsilon_{i_0} \to 0$  as
$\varepsilon \to 0$.\\
\noindent {\bf Case 2: $|x-\overline{x}_{i}(t)|\ge \varepsilon^{\frac13}$
for each $i=1,...,N$}\\
Using in particular that
$$\tilde{\phi}_i = O(\varepsilon^{\frac23})$$
we conclude again from (\ref{eq::12}) that
$$I^\varepsilon \ge \delta/2 \quad \mbox{for}\quad \varepsilon \quad
\mbox{small enough}.$$
Finally case 1 and 2 prove that for $\varepsilon$ small enough we have
$$I^\varepsilon \ge \delta/2 \quad \mbox{for every}\quad x\in\R$$
which ends the proof of the Proposition.\\

\noindent {\bf Proof of Theorem \ref{th::1}}\\
\noindent {\bf Step 1: existence and uniqueness of the solution}\\
First remark that for $\varepsilon$ small enough the initial condition
satisfies
$$-1\le v^\varepsilon_0 \le N+1$$
and the  functions
$$\underline{u}(t,x)=-1-K^\varepsilon t \quad \mbox{and}\quad
\overline{u}(t,x)=N+1+K^\varepsilon t$$
with
$$K^\varepsilon=\varepsilon^{-1}|\sigma|_{L^\infty([0,+\infty)\times \R)}
+\varepsilon^{-2} |W'|_{L^\infty(\R)}$$
are respectively sub and supersolutions of (\ref{eq::1}) on
$[0,+\infty)\times \R$.
From the Perron's method, it follows the existence of a solution
$v^\varepsilon$ of (\ref{eq::1}) on $[0,+\infty)\times \R$ which is moreover unique thanks to the comparison
principle (Proposition \ref{pro:4}).\\
\noindent {\bf Step 2: sub and supersolutions as $\varepsilon \to 0$}\\
Now, given any fixed $T>0$, and thanks to Propositions \ref{pro::2} and
\ref{pro::3}, we can find some $\delta>0$ small enough such that
$\overline{v}^\varepsilon$ is a supersolution of (\ref{eq::1}) on $[0,T)\times
\R$.\\
Similarly we can build a subsolution
$\underline{v}^\varepsilon$ (defined simply as $\overline{v}^\varepsilon$
but with $\delta <0$). At the initial time we have
$$\overline{v}^\varepsilon(0,x) \ge v^\varepsilon(0,x) \ge
\underline{v}^\varepsilon(0,x) \quad \mbox{for all}\quad x\in\R$$
and then from the comparison principle, we get that
$$\overline{v}^\varepsilon(t,x) \ge v^\varepsilon(t,x) \ge
\underline{v}^\varepsilon(t,x) \quad \mbox{for all}\quad t\in [0,T),\quad
x\in\R.$$
Now for fixed $T>0$, we get the result (namely (\ref{eq::c+}) and
(\ref{eq::c-}))
on $[0,T)\times \R$, passing to the limit as
$\varepsilon\to 0$. Finally, because $T>0$  was arbitrary, we recover the
result for all time on $[0,+\infty)\times \R$. This ends the proof of the Theorem.\qed\\


\section{Known and further results on the half-Laplacian}\label{section-half-Laplacian}

In this section, we recall several results that will be used in Section 6 for the proof of Theorems \ref{th::5} and \ref{th::6}. We first recall four definitions (that are equivalent for smooth enough functions) of the half-Laplacian operator
(Fourier, L\'evy-Khintchine, classical harmonic extension, notion of weak solutions).

We will essentially use the notion of weak solution (via the harmonic extension) for which we will show that it coincides with the notion of viscosity solutions (via the L\'evy-Khintchine formula), when the function is in the space $L^\infty(\R)\cap C^{1,\beta}_{loc}(\R)\cap H^{\frac12}(\R)$.

Then we recall results from \cite{Cabre-SolaMorales:layer-solutions}
on the layer solutions and the comparison principle (for the harmonic extension). We also recall regularity results for harmonic extensions and give new useful results for the $L^\infty$ regularity of solutions to half-Laplacian type equations.\\

There are many points of view to handle half-Laplacian operators. Here we have chosen to consider it through a harmonic extension problem, although another more direct approach could have been taken. However, our approach has two advantages: first, many of the results reduce to its analogous for harmonic functions, and second, we use the setting already established by Cabr\'e and Sol\`a Morales in \cite{Cabre-SolaMorales:layer-solutions}. Another advantage is that this is precisely the setting of Cabr\'e-Sire \cite{Cabre-Sire:layer-solutions} for generalization to any power $\alpha\in(0,1)$.

\subsection{The Direct approaches}

\subsubsection{The Fourier approach}

We recall our definition of the Fourier transform for $w \in {\mathcal
  S}(\R)$  (where ${\mathcal S}(\R)$ is the Schwartz space):
$$\widehat{w}(\xi)=\int_{\R}w(x) \  e^{-i\xi x}dx  \quad \mbox{for all}\quad
\xi\in\R$$
and recall the definition of the following Hilbert space
$$H^{\frac12}(\R)=\left\{w\in L^2(\R),\quad \int_{\R}d\xi\
  (1+|\xi|)|\widehat{w}(\xi)|^2 < +\infty \right\}$$
with the scalar product
$$(v,w)_{H^{\frac12}(\R)}=\int_{\R}d\xi\ (1+|\xi|)\widehat{v}^*(\xi)
\widehat{w}(\xi)$$
where $\widehat{v}^*$ denotes here the complex conjugate of $\widehat{v}$.
We also set the following Hilbert space
$$H^{-\frac12}(\R)=\left\{w\in {\mathcal
  S}'(\R),\quad \widehat{w}\in L^2_{loc}(\R),\quad \int_{\R}d\xi\
  (1+|\xi|)^{-1}|\widehat{w}(\xi)|^2 < +\infty \right\}$$
with the scalar product
$$(v,w)_{H^{-\frac12}(\R)}=\int_{\R}d\xi\ (1+|\xi|)^{-1}\widehat{v}^*(\xi)
\widehat{w}(\xi).$$
It is known that $H^{-\frac12}(\R)$ is the dual of $H^{\frac12}(\R)$.
Recall that we set $L=-(-\Delta)^{\frac12}$ to be the half-Laplacian operator
defined in Fourier space by
\begin{equation}\label{eq::3}
\widehat{Lw}(\xi):=-|\xi|\widehat{w}(\xi) \quad \mbox{for}\quad \xi\in\R.
\end{equation}
In particular, $L$ is a continuous linear mapping from $H^{\frac12}(\R)$
into $H^{-\frac12}(\R)$ satisfying
$$\norm{Lw}_{H^{-\frac12}(\R)}\le \norm{w}_{H^{\frac12}(\R)}.$$
Moreover $L$ is self-adjoint as follows: for any $v,w \in H^{\frac12}(\R)$,
we have
$$<Lv,w>_{H^{-\frac12}(\R)\times H^{\frac12}(\R)}=
<Lw,v>_{H^{-\frac12}(\R)\times H^{\frac12}(\R)}.$$
Finally let us remark that for any $w\in H^{\frac12}(\R)$ we have
\be\label{H12-norm}<-Lw,w>_{H^{-\frac12}(\R)\times H^{\frac12}(\R)} +\int_{\R} w^2 = \frac{1}{2\pi}(w,w)_{H^{\frac12}(\R)}.\ee

\subsubsection{The approach in real space}

We have the following result (see \cite{Levy}, or the more recent paper \cite{Droniou-Imbert:fractal-first-order}, Theorem 1):
\begin{theo}\label{th::2}{\bf (The L\'evy-Khintchine formula)}\\
For any $0<\beta<1$, if $w\in H^{\frac12}(\R)\cap C^{1,\beta}_{loc}(\R)$, then we
have
$$Lw(x)=\frac{1}{\pi} \int_{\R}\frac{dz}{z^2}\left(w(x+z)-w(x)-zw'(x) 1_{\left\{|z|\le
      1\right\}}\right).$$
Moreover this formula allows also to define $Lw$ for instance for $w\in L^\infty(\R)\cap C^{1,\beta}_{loc}(\R)$.
\end{theo}

\begin{rem}
In the previous formula the characteristic function $1_{\left\{|z|\le
      1\right\}}$ can be replaced by $1_{\left\{|z|\le
      r\right\}}$ for any $r>0$, which does not change the value of the integral.
\end{rem}

The operator $L$ is called a L\'evy operator associated to the L\'evy
measure $dz/z^2$.
In particular, the L\'evy-Khintchine formula allows to see that for any
$w\in H^{\frac12}(\R)$, we have
\be\label{eq::200}\frac{1}{2\pi}\int_{\R}d\xi\ |\xi||\widehat{w}(\xi)|^2=<-Lw,w>_{H^{-\frac12}(\R)\times
  H^{\frac12}(\R)} = {\frac{1}{2\pi}}
\int_{\R\times \R} dx dx' \ \frac{|w(x')-w(x)|^2}{|x'-x|^2}.\ee
This last formula is a simple exercise and can be found in the book by Lieb-Loss \cite{Lieb-Loss}, theorem 7.12, part ii.

\subsection{The approaches by harmonic extension on $\R^2_+$ and notion of weak solutions}

In the following, we give a characterization of the half-Laplacian operator in $\mathbb R$ through a harmonic extension to the upper half-plane as a Dirichlet-to-Neumann operator. This relation was well known but it has been recently rediscovered by Caffarelli-Silvestre in \cite{Caffarelli-Silvestre}, where they also give the generalization to any fractional power of the Laplacian between zero and one.

The main advantage of this interpretation is that the regularity results and maximum principles for the half-Laplacian can be shown through the analogous results for harmonic functions in the extension, as it has been considered by Cabr\'e and Sol\`a Morales in \cite{Cabre-SolaMorales:layer-solutions}.

\subsubsection{Classical harmonic extension for smooth functions with compact support}

Assume that $w\in C^\infty_{c}(\R)$ (the space of smooth functions with compact support) and consider its harmonic extension $u(x,y)$
on
$$\Omega:=\left\{(x,y)\, : \quad x\in\R,\quad y >0\right\}=\R^2_+$$
defined as the solution of
\begin{equation}\label{eq::4}
\left\{\begin{array}{ll}
\Delta u = 0  &\quad \mbox{on}\quad \Omega,\\
\\
u(x,0)=w(x)  &\quad \mbox{for all}\quad x\in\partial\Omega.
\end{array}\right.
\end{equation}
Then it is known that
\begin{equation}\label{eq::55}
Lw(x)=-\frac{\partial u}{\partial \nu}(x,0) \quad \mbox{for all}\quad
x\in\partial\Omega.
\end{equation}
Here, and for the rest of the paper, $\nu$ denotes the exterior normal along $\partial\Omega$.

The half-Laplacian appears to be a Dirichlet-to-Neumann operator.
This is also called a Steklov-Poincar\'e operator. Moreover, we can use the Poisson kernel $P$ for the upper half-plane to write $u=P*_x w$, i.e.,
\be\label{Poisson-kernel}u(x,y)=c_2\int_{\mathbb R} \frac{y}{\abs{x-\xi}^2+y^2}w(\xi)d\xi,\ee
where $c_2$ is a dimensional constant.

\subsubsection{Harmonic extension of functions in $H^{\frac12}(\R)$}

Now let us define the scalar product
$$(u,v)_H=\int_{\Omega} \nabla u\cdot\nabla v +\int_{\partial \Omega} uv$$
and consider the Hilbert space
$$H=\overline{C^\infty_{c}(\overline{\Omega})}^{|\cdot |_{H}}.$$

Then we have the following result:
\begin{theo}\label{th::3}{\bf (The trace mapping)}\\
The trace mapping
$$\begin{array}{llcl}
T: & H & \longrightarrow & H^{\frac12}(\R)\\
 & u & \longmapsto & Tu :=u(\cdot,0)
\end{array}$$
is linear continuous and satisfies
\be\label{trace-mapping}\norm{Tu}_{H^{\frac12}(\R)}^2 \le 2\pi\norm{u}_{H}^2.\ee
Moreover we have $\norm{Tu}^2_{H^{\frac12}(\R)} = 2\pi\norm{u}^2_{H}$ if and only if $u$ is
harmonic on $\Omega$.
\end{theo}

\begin{proof}
For $u$ smooth and with compact support, the classical embedding trace inequality reads
\be\label{trace-inequality}
\norm{Tu}^2_{\dot{H}^{\half}(\R)}:=\int_{\mathbb R\times \mathbb R}dxdx' \ \frac{|Tu(x')-Tu(x)|^2}{|x'-x|^2}\leq 2\pi \int_\Omega \abs{\nabla u}^2\,dxdy.\ee
A proof of the previous inequality can be found in corollary 6.3 in \cite{Alberti-Bouchitte-Seppecher:Phase-transition}.
Then \eqref{trace-mapping} follows by putting together the previous line with \eqref{H12-norm} and \eqref{eq::200}. For a general $u$ the result follows by density, while the equality case is dealt with in Theorem \ref{th::4} below.\\
\end{proof}


Now, we can reformulate problem \eqref{eq::4}-\eqref{eq::55} in a more
general setting:
\begin{theo}\label{th::4}{\bf (Half-Laplacian defined by the harmonic extension)}\\
For any $w\in H^{\frac12}(\R)$, there exists a unique $u\in H$ which is
harmonic on $\Omega$, such that $Tu=w$, and it is written as $u=P*_x w$. Moreover for any $v\in H$, we have
\be\label{weak-notion}<-Lw,Tv>_{H^{-\frac12}(\R)\times H^{\frac12}(\R)} = \int_{\Omega} \nabla
u\cdot \nabla v.\ee
In particular,
$$\frac{1}{2\pi}\int_{\R} d\xi\ |\xi| |\widehat{w}(\xi)|^2 = <-Lw,w>_{H^{-\frac12}(\R)\times H^{\frac12}(\R)} =\int_{\Omega} |\nabla u|^2,$$
and
\be\label{equality-norms}\int_\Omega\abs{\nabla u }^2+\int_{\partial\Omega} w^2=\frac{1}{2\pi}\norm {w}_{H^{\frac{1}{2}}(\mathbb R)}^2.\ee
\end{theo}

The previous theorem allows to define weak solutions (in the sense of distributions) for half-Laplacian equations, that may not be necessarily equivalent to the notion of viscosity solutions introduced in Definition \ref{defi::1}. Although the introduction of a new concept may seem confusing at first, it will allow to use the available functional analysis tools. Moreover, this definition is a particular case of the one introduced in \cite{Cabre-SolaMorales:layer-solutions}. Thus we set:

\begin{defi} \label{defi-weak-solution}{\bf (Weak sub/super/solution)}\\
Let $d$ be a measurable bounded function on $\mathbb R$, and let $g\in L^2(\mathbb R)$.
We say that $w\in H^\half(\mathbb R)$ is a weak supersolution for the equation
\begin{equation}\label{weak-equation}
Lw -d(x)w=g \quad \mbox{ on }\partial\Omega
\end{equation}
if, for $u:=P*_x w$ and
$$B(u,v):=\int_{\Omega}\nabla u\cdot\nabla v+\int_{\partial\Omega} d(x)(Tu)(Tv),$$
we have that
\be\label{weak-super-solution}B(u,v)\geq \int_{\partial\Omega} gTv\quad\mbox{for all }v\in H.\ee
Respectively, $w$ is a weak subsolution if
\be\label{weak-sub-solution} B(u,v)\leq \int_{\partial\Omega} gTv\quad\mbox{for all }v\in H.\ee
Accordingly, we say that $w$ is a weak solution if it is both a weak sub and supersolution.\\
\end{defi}

We summarize the different concepts in the following lemma:

\begin{lem}\label{lemma-equivalence}{\bf (Equivalence of solutions)}\\
Given $w\in L^{\infty}(\mathbb R)\cap \mathcal C^{1,\beta}_{\mbox{loc}}(\mathbb R)$, the following two ways of compuing $Lw$, i.e,
\begin{enumerate}
\item as a Dirichlet to Neumann operator (see section 5.3.1., in particular \eqref{eq::4}-\eqref{eq::55});

\item using the L\'evy-Khintchine formula (see theorem \ref{th::2}),
\end{enumerate}
are equivalent.

If, in addition, $w\in H^\half(\mathbb R)$, then the half-Laplacian definition in term of distributions as given in theorem \ref{th::4} is also equivalent. As a consequence, the notions of viscosity and weak solutions from Definitions \ref{defi::1} and \ref{defi-weak-solution}, respectively, coincide in the smooth case.
\end{lem}

\begin{proof}
We refer the reader to the paper \cite{Droniou-Imbert:fractal-first-order}.
\end{proof}

\subsection{Miscellanea}

Here we remind the reader some basic facts.

For harmonic extensions, we have (cf. \cite{Stein-Weiss:Fourier-analysis}, chapter II.2):

\begin{theo}[\textbf{$L^p$ bound}]\label{thm-Stein}
If $w\in L^p(\mathbb R)$, $1\leq p\leq \infty$, and $u=P*_x w$ is the Poisson integral of $w$, then $u$ is harmonic in $\mathbb R^2_+$, $\lim_{y\to 0} u(x,y)=w(x)$ for almost every $x\in\mathbb R$ and
$$\norm{u(\cdot,y)}_{L^p(\mathbb R)}=\lp\int_{\mathbb R}\abs{u(x,y)}^p\;dx\rp^{1/p}\leq \norm{w}_{L^p(\mathbb R)}$$
for all $y>0$.
If $w$ is bounded on $\mathbb R$, then also $u$ is uniformly bounded in $\overline{\mathbb R^2_+}$.\\
\end{theo}

As a consequence, given $w\in H^\half(\mathbb R)$, then for $u=P*_x w\in H$ we have that $u\in W^{1,2}(B_R^+)$ for every half ball $B_R^+$, since
$$\int_{B_R^+} u^2\,dxdy\leq \int_{y\in (0,R)}\int_{x\in\mathbb R} u^2\,dxdy\leq R \int_{\partial\Omega} w^2.$$

Moreover (c.f. \cite{Evans}, Chapter 2, Section 2, Theorem 7):
\begin{theo}[\textbf{Local gradient estimates}]\label{local-grad-estimates}
Assume that $u$ is harmonic in a domain $U$. Then
$$\abs{D^\alpha u(x_0)}\leq \frac{C_k}{r^{k}} \norm{u}_{L^{\infty}(B_r(x_0))},$$
for every ball $B_r(x_0)\subset U$, and multiindex $\abs{\alpha}=k$.
\end{theo}

Next, we give some classical facts on Sobolev spaces, from \cite{Mazja} (Section 7.1.3, Theorem 4, part (6) and Section 8.8.1, Theorem 4):

\begin{theo}\label{Sobolev}{\bf (Sobolev embedding)}
The injection
$$H^{\frac{1}{2}}(\mathbb R)\hookrightarrow L^p(\mathbb R)$$
is continuous, for any $p>1$.
Moreover, if $K\subset \mathbb R$ is a compact subset, then  the restriction map
$$H^{\frac{1}{2}}(\mathbb R)\hookrightarrow L^p(K)$$
is a compact embedding for any $p>2$.
\end{theo}

\subsection{Layer solutions for half-Laplacian problems}

The paper \cite{Cabre-SolaMorales:layer-solutions} is concerned with the boundary reaction problem
\begin{equation}\label{Cabre-SolaMorales-layer}
\left\{\begin{split}
\Delta u &= 0 &\quad \mbox{in } \Omega,\\
\frac{\partial u}{\partial \nu} &= f(u) &\quad \mbox{on } \partial\Omega,
\end{split}\right.\end{equation}
for some non-linearity $f$.
We say that $u$ is a \emph{layer solution} of \eqref{Cabre-SolaMorales-layer} if it satisfies \eqref{Cabre-SolaMorales-layer}, $u_x>0$ on $\partial\Omega$ and
$$\lim_{x \to -\infty} u(x,0)=0,\quad \lim_{x\to +\infty} u(x,0)=1.$$

Since we are only concerned with smooth solutions of \eqref{Cabre-SolaMorales-layer}, then we do not need to worry here about the different concepts of weak or viscosity solutions for half-Laplacian problems previously introduced.

In the next theorem we summarize the basic facts about existence and behavior of layer solutions for \eqref{Cabre-SolaMorales-layer} (see theorems 1.2, 1.4, 1.6, and equation (2.27) from \cite{Cabre-SolaMorales:layer-solutions}).

\begin{theo}{\bf (Known existence and properties of the layer solution)}\label{thm-layer-solutions}\\
Let $f$ be any $\mathcal C^{1,\beta}$ function with $0 < \beta < 1$ and such that $W'= -f$. Then there exists a layer solution $u_W\in \mathcal C^{2,\beta}_{\mbox{loc}}(\overline\Omega)$ of \eqref{Cabre-SolaMorales-layer} if and only if
$$W'(0)=W'(1)=0\quad\mbox{and}\quad W> W(0)=W(1)\quad\mbox{in } (0,1).$$
Moreover,
\begin{enumerate}
\item If $W''(0), W''(1)>0$, then a layer solution of \eqref{Cabre-SolaMorales-layer} is unique up to translations in the $x$ variable, and it satisfies
\be\label{derivative-positive}(u_W)_x>0\quad\mbox{in }\overline{\Omega}.\ee
\item We have the estimate
\be\label{gradient-bound}\abs{\nabla u_W(x,y)}\leq \frac{C}{1+\abs{(x,y)}}, \quad \mbox{for all }x\in\mathbb R, y\geq 0.\ee
\item Set $\phi(x):=u_W(x,0)$. We have the bounds
\begin{equation}
\label{growth-not-sharp}\abs{\phi(x)-H(x)}\leq \frac{C}{1+\abs{x}} \quad\mbox{for all }x\in\mathbb R,
\end{equation}
\be\label{derivative-bound}
\frac{C_1}{1+\abs{x}^2} \leq \phi'(x)\leq \frac{C_2}{1+\abs{x}^2},\quad \mbox{for all }x\in\mathbb R,
\ee
for some positive constants $C,C_1,C_2$.
\item A layer solution is always a stable solution. This means that
\be\label{stable-solution}\int_{\Omega}\abs{\nabla v}^2+\int_{\partial\Omega}W''(\phi)v^2\geq 0\ee
for every function $v\in\mathcal C^1(\overline{\Omega})$ with compact support in $\overline{\Omega}$.\\
\end{enumerate}
\end{theo}

One of the crucial ingredients in the present work is the precise knowledge of the asymptotics for $\phi$ as given in \eqref{eq::6}. In order to improve from the known \eqref{growth-not-sharp}, we will use a comparison argument with the explicit layer solution of the Peierls-Nabarro problem (see \cite{Cabre-SolaMorales:layer-solutions}, Lemma 2.1, for the proofs, or the original paper \cite{Toland:Peierls-Nabarro}): for each $a>0$, consider the equation
\begin{equation}\label{problem-model}\left\{\begin{split}
&\Delta u^a= \;0 \quad &\mbox{in } \Omega,\\
&\frac{\partial u^a}{\partial \nu}= -V_a'(u^a) \quad &\mbox{on } \partial\Omega,
\end{split}\right.
\end{equation}
with potential given by
$$V_a'(t)=-\frac{1}{2\pi a}\sin\left[2\pi\lp t-\tfrac{1}{2}\rp\right].$$
It is well known that the function
\begin{equation}\label{model-layer}
u^a(x,y)=\frac{1}{\pi}\arctan\frac{x}{y+a}+\frac{1}{2}
\end{equation}
is a layer solution for the problem \eqref{problem-model} and, in fact, it is the only one.\\

A simple computation gives that
\begin{equation}\label{computation-layer}
u^{a}_x\leq a\frac{\partial u_x^{a}}{\partial\nu}+2u_x^{a}\end{equation}
along $\partial\Omega$.
Set $\phi^a(x):=u^a(x,0)$.
We  know explicitly its asymptotic behavior
\begin{equation}\label{asymptotic-layer}
\left|\phi^a(x)-H(x)+\frac{a}{\pi x}\right| \quad \le \quad \frac{C}{1+\abs{x}^2}   \quad \mbox{for all }x\in\mathbb R,
\end{equation}
an improvement from \eqref{growth-not-sharp}.\\

Finally, fix $\bar u:=(u_W)_x\in \mathcal C^{1,\beta}_{\mbox{loc}}(\bar \Omega)$, and note that $\bar u=P*_x\phi'$. Then:

\begin{lem}\label{properties}{\bf (Properties of $\phi'$)}. For $\bar u$ defined as $\bar u=P*_x \phi'$, we have that
\begin{enumerate}
\item $\bar u\in H$.
\item $\phi'$ is a weak solution of $Lw-W''(\phi)w=0$ in the sense of Definition \ref{defi-weak-solution}, i.e.,
\begin{equation}\label{eq::907}
B(\bar u,v)=0 \quad \mbox{for all } v\in H.
\end{equation}
    \end{enumerate}
\end{lem}

\begin{proof}
We recall estimate \eqref{eq::5}. It shows that $\phi'\in L^2(\mathbb R)$. To complete the proof of the first claim we need to check that $\int_{\Omega} \abs{\nabla \bar u}^2<\infty$.\\
\noindent \emph{Step 1: $\int_{\left\{0<y<R\right\}}|\nabla \bar u|^2 < +\infty$.}\\
We proceed as in the proof of Lemma 2.3 in \cite{Cabre-SolaMorales:layer-solutions}. We first set
$$v(x,y)=\int_0^y \bar u(x,t)dt.$$
We can see that $v$ is a (smooth) solution of
$$\left\{\begin{array}{ll}
\Delta v=W''(\phi(x))\phi'(x) &\quad \mbox{in } \Omega, \\
v(x,0)=0 & \quad\mbox{on }\partial\Omega.
\end{array}\right.$$
We know (\ref{gradient-bound}) which implies that $v\in L^2\left(\R\times [0,R]\right)$.
Because the right hand side of the equation is bounded in $L^2(\mathbb R\times [0,R])$, and here we are using again the precise decay of $\phi'$, then we get also that $v\in W^{2,2}_{loc}$ using the boundary gradient estimates for harmonic functions in \cite{Gilbarg-Trudinger}, Lemma 9.12. As a consequence, $\bar u\in W^{1,2}(\R\times [0,R])$.\\

\noindent \emph{Step 2: $\int_{\left\{y>R\right\}}|\nabla \bar u|^2 < +\infty$.}\\
Far away from the boundary of $\Omega$, we can use standard interior gradient estimates (Lemma \ref{local-grad-estimates}) for the harmonic function $\bar u$, say, at a ball $B_{1}(x,y)\subset \Omega$ so that
$$\abs{\nabla \bar u(x,y)}\leq C \norm{\bar u}_{L^\infty (B_{1}(x,y))}\le
\frac{C'}{1+(|(x,y)|-1)^2}$$
where we have used (\ref{gradient-bound}) in the last inequality.\\
\noindent \emph{Step 3: Conclusion.}\\
From Steps 1 and 2, we get that $\bar u\in H$.\\
\noindent \emph{Step 4: Proof of (\ref{eq::907}).}\\
From (\ref{Cabre-SolaMorales-layer}) with $f=-W'$, we get for any $w\in C^\infty_c(\overline{\Omega})$:
$$0=\int_{\Omega}\nabla  u\cdot\nabla w +\int_{\partial \Omega} wW'(u)$$
Setting $w=v_x$ with $v\in C^\infty_c(\overline{\Omega})$, we get by integration by parts
$$B(\bar u,v)=0$$
Finally, we conclude by density of $C^\infty_c(\overline{\Omega})$ in $H$.
\end{proof}
\qed

\subsection{Comparison principles for smooth enough solutions}

Next, we consider some comparison results for half-Laplacian equations:

\begin{lem}[{\bf Strong comparison principle}, Lemma 2.8 in \cite{Cabre-SolaMorales:layer-solutions}]\label{lema-comparison}
Let $v$ be a bounded function in $\Omega$,  $v\in\mathcal C^2(\Omega)$ and $\mathcal C^1$ up to the boundary $\partial\Omega$,  satisfying
\begin{equation}\label{equation-max-principle}
\left\{\begin{split}
&-\Delta v \geq 0 \quad &\mbox{ in } \Omega,\\
&\frac{\partial v}{\partial \nu} +d(x)v\geq 0 &\mbox{ on }\partial\Omega,
\end{split}\right.\end{equation}
for some bounded function  $d$ satisfying $d(x)\geq \tau$ for all $x\in \mathbb R$ and some $\tau>0$. Then $v>0$ in $\overline{\Omega}$ unless $v\equiv 0$.
\end{lem}

We will also need the following variation:

\begin{cor}\label{cor-comparison}
Let $v$ be a bounded function in $\Omega$, $v\in\mathcal C^2(\Omega)$ and $\mathcal C^1$ up to the boundary $\partial\Omega$, satisfying
\begin{equation*}\left\{\begin{split}
&\Delta v = 0 \quad &\mbox{ in } \Omega,\\
&\frac{\partial v}{\partial \nu} +d(x)v=M(x) &\mbox{ on }\partial\Omega,
\end{split}\right.\end{equation*}
where $d$ is a bounded function satisfying $d(x)\geq \tau$ for all $x\in \mathbb R$ and some $\tau>0$. Assume, in addition, that the right hand side $M$ can be estimated by
$$\abs{M(x)}\leq \frac{C}{1+\abs{x}^2},\quad x\in\mathbb R$$
for some constant $C\geq 0$.
Then there exists a non-negative constant $\tilde C$ depending on $C,\tau$ such that
$$\abs{v(x,0)}\leq \frac{\tilde C}{1+\abs{x}^2}$$
 for all $x\in\mathbb R$.
\end{cor}

\begin{proof}
Fix $a:=2/\tau$, and consider the model layer solution $u^a$ given in \eqref{model-layer}, with trace $\phi^a$.  Because of hypothesis on $M$ and  the decay of $(\phi^a)_x$ from \eqref{derivative-bound}, we can find a positive constant $K$ such that $\abs{M(x)}\leq K \phi^{a}_x(x)$ for all $x\in\mathbb R$. Next, compute
$$\frac{\partial v}{\partial \nu}+d(x) v=M(x)\leq K u^{a}_x\leq K\left[\frac{2}{\tau}\frac{\partial u_x^{a}}{\partial\nu}+2u_x^{a}\right]\leq
\frac{2K}{\tau} \left[\frac{\partial u_x^{a}}{\partial\nu}+d(x) u_x^{a}\right]\quad\mbox{along }\partial\Omega,$$
where we have used \eqref{computation-layer} in the second inequality.
Now we apply Lemma \ref{lema-comparison} to the difference $\left[\frac{2K}{\tau}u_x^{a}-v\right]$ to get $v\le \frac{2K}{\tau}u_x^{a}$ in $\bar\Omega$. The same argument, applied to $-v$ gives that $\abs{v}\le \frac{2K}{\tau}u_x^{a}$. Taking into account the estimate \eqref{derivative-bound}, we achieve the conclusion of the corollary.
\end{proof}

\qed

\subsection{Further properties of weak solutions}

The aim of this subsection is to extend the classical results for harmonic functions to the boundary $\partial\Omega$, in order to obtain regularity and maximum principles for weak solutions (in the sense of definition \ref{defi-weak-solution}). We set
$$u^-:=-\inf \{u,0\},\quad u^+:=\sup\{u,0\}.$$

Our first result is a weak maximum principle for non-smooth solutions:

\begin{pro}\label{prop-weak-comparison}{\bf (Weak maximum principle).}
Let $u\in H$ be a weak supersolution of \eqref{weak-equation} with $g\equiv 0$, and let $d(x)$ be a measurable bounded function. Set $u=P*_x w$. Assume that $u\geq 0$ in $\overline{B_R^+}$, and that $d>0$ on $\mathbb R\backslash[-R,R]$. Then
$$u\geq 0 \quad\mbox{in }\mathbb R^2_+,$$
and also, $Tu\geq 0$ on $\mathbb R$.
\end{pro}

\begin{proof}
This is a small variation of the classical proof for the Laplacian, that can be found in Gilbarg-Trudinger \cite{Gilbarg-Trudinger}, Theorem 8.1. Note that $(Tu)^-= T(u^-)$ a.e. in $\mathbb R$.
Set
$$\Omega^-:=supp (u^-)\quad\mbox{and}\quad\Theta:=supp((Tu)^-)\subset \mathbb R.$$
We know that $\Omega^-$ is contained in $\mathbb R^2_+\backslash B_R^+$ and $\Theta\subset \mathbb R\backslash(-R,R)$.  Use $u^-\in H$ as an admissible test function in \eqref{weak-super-solution}; then
$$0\leq -\int_{\Omega^-} \abs{\nabla u}^2-\int_{\Theta}d(x)u^2.$$
The proof is easily completed.
\end{proof}
\qed

Note that a function $u\in H$ must have trace $Tu$ belonging to $L^p(\mathbb R)$ for all $p>1$. However, it may not belong to $L^\infty(\mathbb R)$. In the next result, we use a Moser iteration scheme to show that weak solutions are actually bounded.

\begin{pro}\label{prop-moser-iteration}{\bf ($L^\infty$ bound for $g=0$).}
Let $w\in H^{\half}(\mathbb R)$ be a weak subsolution of \eqref{weak-equation} with $g\equiv 0$ for some uniformly bounded  function $d$ defined on $\mathbb R$.
Then
$$\sup_{\mathbb R} w\leq C \norm{w^+}_{L^2(\mathbb R)}.$$
If $w$ is a supersolution, the conclusion reads
$$\sup_{\mathbb R} \{-w\}\leq C \norm{w^-}_{L^2(\mathbb R)}.$$
The constant does not depend on $w$.
As a corollary, $u:=P*_x w$ is also uniformly bounded on $\Omega$.
\end{pro}

\begin{proof} This is the classical proof by Moser for the Laplacian case. We follow closely the arguments from Theorems 8.15 and 8.18 in Gilbarg-Trudinger \cite{Gilbarg-Trudinger}, so we will not give every single detail here.
Assume that $w$ is a subsolution, and set $u:=P*_x w$. Given $\sigma\ge 1$, we insert into \eqref{weak-sub-solution}, the test function
$$v=(u^+)^\sigma.$$
Note that $Tv=(Tu)^\sigma$ a.e.. Here we point out that this $v$ might not belong to $H$; in order to make the argument rigorous, a truncated version must be used instead in the definition of $v$, say replacing $u^+$ by
\begin{equation}\label{eq::902}
\min\{u^+,n\},\quad \mbox{for}\quad n\in \mathbb N.
\end{equation}
We obtain
$$\sigma\int_{\Omega} u^{\sigma-1}\nabla u\cdot\nabla u^+\leq \int_{\partial\Omega} -d(x) (u^+)^{\sigma}u.$$
Next, we set $u_1:=(u^+)^\gamma$ and $w_1:=Tu_1$ for $\gamma=\frac{\sigma+1}{2}$, and use that $|d(x)|\le C_0$. Thus we get
\be\label{moser-1}\int_{\Omega} \abs{\nabla u_1}^2\leq C_1(\gamma)\int_{\partial\Omega} w_1^2
\quad \mbox{with}\quad C_1(\gamma)=\frac{C_0\gamma^2}{2\gamma-1}.\ee

For the left hand side, we proceed as follows.
We use the trace inequality \eqref{trace-inequality} to estimate
$$\frac{1}{2\pi}\norm{w_1}^2_{\dot H^{\half}(\mathbb R)} \le \int_{\Omega} \abs{\nabla  u_1}^2.$$
But, on the other hand, by Theorem \ref{Sobolev}, for every $p>1$, there exists a constant $C_p>0$ such that
$$\begin{array}{ll}
||w_1||^2_{L^{2p}}&\le C_p\left(||w_1||^2_{L^2}+||w_1||^2_{\dot{H}^{\frac12}}\right)\\
\\
&\le C_2(\gamma) ||w_1||^2_{L^2} \quad \mbox{with}\quad C_2(\gamma)=C_p(1+2\pi C_1(\gamma)),
\end{array}$$
where we have used \eqref{moser-1} in the last line.\\
If we denote
$$\Psi (q):=\lp\int_{\partial\Omega} \abs{w^+}^{q}\rp^{\frac{1}{q}},$$
then we have shown that
$$\Psi(2p\gamma)\leq \left[C_2(\gamma)\right]^{\frac{1}{2\gamma}}\Psi(2\gamma).$$
Fixed $p>1$ and choose $\gamma_m=p^m$ for $m\in\N$; by iteration we know
\be\label{iteration-step}\Psi(2\gamma_m)\leq C_m\Psi(2),\quad \forall m\in\mathbb N\backslash\left\{0\right\},\quad C_m:=\prod_{j=0}^{m-1}  \left[C_2(\gamma_j)\right]^\frac{1}{2\gamma_j}.\ee
It is then easy to check that $C_m\le C_\infty$ uniformly bounded in $m\in\mathbb N$ (which is also shown, for instance, in page 190 of \cite{Gilbarg-Trudinger}).
Remark that
$$\sup_{(-R,R)} \abs{w^+}=\lim_{q\to\infty} \lp\frac{1}{2R} \int_{(-R,R)} \abs{w^+}^q\rp^{\frac{1}{q}}.$$
Letting $m$ to infinity in \eqref{iteration-step}, we have
$$\sup_{(-R,R)} \abs{w^+}\leq C_\infty  \norm{w^+}_{L^{2}(\mathbb R)},$$
for every $R>0$. We get the first part of the conclusion, taking $R\to +\infty$ and then $n\to +\infty$ in (\ref{eq::902}).
The second part is analogous, using $u^-$ instead.
The final assertion is an immediate consequence of Theorem \ref{thm-Stein}.
\end{proof}
\qed

We will use also the following particular case (that is not the optimal one can prove, but it suffices for our purposes):

\begin{cor}\label{cor-Moser}{\bf ($L^\infty$ bound when $g$ is bounded).}
Let $w\in H^{\half}(\mathbb R)$ be a weak solution of \eqref{weak-equation} with non-vanishing right hand side $g$.
If there exists $p_0>1$ such that
\begin{equation}\label{eq::904}
\norm{g}_{L^p(\mathbb R)}\leq c \quad \mbox{for all}\quad p\ge p_0
\end{equation}
(with $c$ independent on $p$), then $w$ is bounded.
\end{cor}

\begin{proof}
We use Proposition \ref{prop-moser-iteration} with some modifications in order to account for the extra right hand side term involving $g$. By hypothesis, $w$ is a weak solution satisfying for all  $v\in H$,
$$B(u,v)=-\int_{\partial\Omega} gv, \quad\mbox{for }u=P*_x w.$$
Similarly as  before (using if necessary a truncated version), we use the test function $v=\abs{u}^{\sigma-1} u$,
$u_1=|u|^{\gamma-1}u$ and $w_1 = Tu_1$. We use H\"older inequality with exponents $2\gamma$, $\frac{2\gamma}{2\gamma-1}$, which gives
$$\abs{\int_{\partial\Omega} gv}\leq \lp\int_{\partial\Omega} \abs{g}^{2\gamma}\rp^{\frac{1}{2\gamma}} \lp\int_{\partial\Omega} w_1^{2}\rp^{\frac{2\gamma-1}{2\gamma}},$$
we get
\begin{equation*}\begin{split}
\int_{\Omega} \abs{\nabla u_1}^2 &\le C_1(\gamma) \int_{\partial\Omega} w_1^{2} +\tilde{C}_1(\gamma)\lp\int_{\partial\Omega} w_1^{2}\rp^{\frac{2\gamma-1}{2\gamma}}\quad \mbox{with}\quad  \tilde{C}_1(\gamma)= \frac{\gamma^2}{2\gamma-1}\cdot \left(\int_{\partial \Omega} |g|^{2\gamma}\right)^{\frac{1}{2\gamma}}.
\end{split}\end{equation*}
Assuming (\ref{eq::904}) with $2\gamma\ge p_0$, we get
$$\int_{\Omega} \abs{\nabla u_1}^2 \le \bar{C}_1(\gamma) \max\left\{1,\int_{\partial \Omega} w_1^2\right\} \quad \mbox{with}\quad \bar{C}_1(\gamma) = C_1(\gamma) + \frac{c\gamma^2}{2\gamma-1}$$

The rest of the argument follows similarly, with $\bar{C}_2(\gamma)=\max\left\{1,C_p(1+2\pi \bar{C}_1(\gamma))\right\}$, $\bar\Psi(q)=\max \left\{1,\left(\int_{\partial\Omega}|w|^q\right)^{\frac1q}\right\}$, and $\bar \Psi(2p\gamma)\le (\bar C_2(\gamma))^{\frac{1}{2\gamma}}\bar \Psi(2\gamma)$, starting the iteration at $2\gamma=p_0$.\qed\\
\end{proof}

We now quote the regularity theory for bounded weak solutions of \cite{Cabre-SolaMorales:layer-solutions}.

\begin{pro}{\bf (Regularity of weak solutions)}. \label{prop-regularity}
Let $R>0$, $\beta\in(0,1)$, and let $w\in H^{\half}(\mathbb R)$ be a weak solution of \eqref{weak-equation} with $d\in\mathcal C^\beta([-4R,4R])$, $g\in\mathcal C^\beta([-4R,4R])$. If $u=P*_x w \in L^\infty (B^+_{4R})\cap W^{1,2}(B_{4R}^+)$, then $u\in \mathcal C^{1,\beta}(\overline{B_R^+})$ and
$$\norm{u}_{\mathcal C^{1,\beta}(\overline{B_R^+})}\leq C_R,$$
for some constant $C_R$ that depends only on $\beta$, $R$ and on the upper bounds for $\norm{u}_{L^\infty (B^+_{4R})}$, $\norm{d}_{\mathcal C^\beta([-4R,4R])}$ and $\norm{g}_{\mathcal C^\beta([-4R,4R])}$.
\end{pro}

\begin{proof}
This is a consequence of Lemma 2.3 from \cite{Cabre-SolaMorales:layer-solutions}, using  definition \ref{defi-weak-solution} for weak solutions.
\end{proof}
\qed

In particular, we obtain $\mathcal C^{1,\beta}_{\mbox{loc}}$  regularity up to the boundary $\partial\Omega$:

\begin{cor}{\bf (Uniform regularity of weak solutions)}.\label{cor-regularity}
Given $R$, $\beta$, $w$, $u$, $d$, $g$ as above, assume, in addition, that  $\norm{d}_{\mathcal C^\beta(\mathbb R)}\leq C$, $\norm{g}_{\mathcal C^\beta(\mathbb R)}\leq C$. Then $u\in \mathcal C^{1,\beta}(\mathbb R\times [0,R])$ and
$$\norm{\nabla u}_{L^\infty(\mathbb R\times [0,R])}\leq C_R,$$
for some constant $C_R$ that depends only on $\beta$, $R$ and on the upper bounds for $\norm{u}_H$, $\norm{d}_{\mathcal C^\beta(\mathbb R)}$ and $\norm{g}_{\mathcal C^\beta(\mathbb R)}$.
\end{cor}

\begin{proof}
We apply the previous proposition to every ball $B_R(x_0,0)$, $x_0\in\mathbb R$, as in Lemma 2.3,
part (b) of \cite{Cabre-SolaMorales:layer-solutions}.
It uses the  condition that $u$ is uniformly bounded on $\Omega$, which in our case follows from Corollary
\ref{cor-Moser}.
\end{proof}
\qed


\section{Results on the layer solution and the corrector}\label{section-corrector}

In this last section, we give the proofs of Theorems \ref{th::5} and \ref{th::6}.

\subsection{Proof of Theorem \ref{th::5}}

Let $W$ be a potential satisfying hypothesis \emph{Ai)}. Theorem \ref{thm-layer-solutions} shows existence, regularity, uniqueness and asymptotic behavior of a layer solution $0<u_W<1$ satisfying $u_W(0,0)=1/2$, $u\in\mathcal C^{2,\beta}_{\mbox{loc}}(\overline{\Omega})$, for the problem
\begin{equation}\label{smooth-solution}
\left\{\begin{split}
\Delta u &= 0 &\quad \mbox{in } \Omega,\\
\frac{\partial u}{\partial \nu} &= -W'(u) &\quad \mbox{on } \partial\Omega.
\end{split}\right.\end{equation}
We denote its trace value by $\phi:=u_W(\cdot,0)$. The aim of the next paragraphs is to show the more precise estimate \eqref{eq::6} for the asymptotic behavior near infinity.\\

Set $a=1/\alpha$, where $\alpha=W''(0)$ as defined in (\ref{eq::900}), and choose the model layer solution $u^a$ as given in \eqref{model-layer}. We know explicitly because of \eqref{asymptotic-layer} that $\phi^a:=u^a(\cdot,0)$ has the right asymptotic behavior \eqref{eq::6}. We would like to compare $u_W$ to $u^a$, in order to obtain the analogous estimate for $u_W$. Therefore, we set $v:=u_W-u^a$. Then $v$ satisfies
\begin{equation*}\left\{\begin{split}
&\Delta v= \;0 \quad &\mbox{in } \Omega,\\
&\frac{\partial v}{\partial \nu}= -W'(\phi)+V'_a(\phi^a)\quad &\mbox{on } \partial\Omega.
\end{split}\right.
\end{equation*}
Now we estimate
$$-W'(\phi)+W'(\phi^a)=-W''(0)[\phi-\phi^a]+O\lp \frac{1}{1+\abs{x}^2}\rp,$$
thanks to the growth estimate \eqref{growth-not-sharp}. And also,
$$-W'(\phi^a)+V'_a(\phi^a)=[-W''(0)+V_a''(0)]\phi^a+O\lp \frac{1}{1+\abs{x}^2}\rp.$$
Because of the choice of $a$, we have that $W''(0)=V_a''(0)$. Thus
 adding the previous two lines we conclude that
$v$ is a solution of
\begin{equation*}\left\{\begin{split}
&\Delta v = 0 \quad &\mbox{
 in } \Omega,\\
&\frac{\partial v}{\partial \nu} +W''(0) v=M(x) &\mbox{ on }\partial\Omega,
\end{split}\right.\end{equation*}
for some function $M$ satisfying
$$\abs{M(x)}=O\lp \frac{1}{1+\abs{x}^2}\rp .$$
Therefore, \eqref{eq::6} follows from Corollary \ref{cor-comparison} (comparison principle for smooth solutions) and the fact that $W''(0)>0$.

The proof of Theorem \ref{th::5} is completed due to the fact that a smooth, bounded solution $u$ of \eqref{smooth-solution}, with trace $\phi$, is a viscosity solution in the sense of Definition \ref{defi::1} of
$$L \phi -W'(\phi)=0\quad \mbox{in }\mathbb R,$$
with $L$ defined through the L\'evy-Khintchine formula, as it is noted in Lemma \ref{lemma-equivalence}.
\qed

\subsection{Proof of Theorem \ref{th::6} - coercivity}

Here we remind the reader of the notation introduced in Section \ref{section-half-Laplacian}: let $\phi$ be the layer solution constructed Theorem \ref{th::5}, with harmonic extension $u_W$ (some properties are given also in Theorem \ref{thm-layer-solutions}). Its asymptotic behavior is given in Theorem \ref{th::5}:
\be\label{asymptotics}\left\{\begin{split}
\left|\phi(x)-H(x)+\frac{1}{\alpha \pi x}\right| \quad &\le \quad \frac{C}{1+x^2}, \\
\abs{\phi'(x)} &\leq \frac{C}{1+x^2},
\end{split}\right.\ee
for all $x\in\mathbb R$. We also fix  $\bar u:=(u_W)_x\in \mathcal C^{1,\beta}_{\mbox{loc}}(\bar \Omega)$, and note that $\bar u=P*_x\phi'\in H$. For some of its properties, see also Lemma \ref{properties}.\\

In the following, we set up the necessary functional analysis ingredients for Theorem \ref{th::6}, while the proof will be completed in the last subsection.

We set $H_0$ to be the subspace in $H$ consisting of functions whose trace is orthogonal to $\phi'$ in $L^2(\partial\Omega)$, i.e,
\begin{equation}\label{space}
H_0:=\left\{ u\in H : \int_{\partial\Omega} u\phi'=0\right\}.
\end{equation}

\begin{lem}
$H_0$ is a closed subspace of $H$.
\end{lem}
\begin{proof}
As we have noted in Theorem \ref{th::3}, the trace of functions in $H$ is a well defined function in $H^{\frac{1}{2}}(\mathbb R)$. In particular, the trace mapping $T:H\to L^2(\mathbb R)$ is continuous by Sobolev embedding (Theorem \ref{Sobolev}). Then, if we have a sequence $\{u_k\}_{k=1}^\infty\to u_0$ in $H$ with $u_k\in H_0$, then
$$\abs{\int_{\partial\Omega} \phi'(u_k-u_0)}\leq \lp\int_{\partial\Omega} \phi'^2\rp^{1/2} \lp\int_{\partial\Omega} \abs{u_k-u_0}^2\rp^{1/2}\to 0\quad\mbox{when } k\to \infty,$$
using the bounds for $\phi'$ given in \eqref{asymptotics}.
We conclude that $u_0\in H_0$, and the lemma is shown.
\end{proof}
\qed

Now, given $u,v\in H$, we define the bilinear functional
\be\label{bilinear-functional}
B(u,v):=\int_\Omega \nabla u \cdot\nabla v+\int_{\partial\Omega} W''(\phi)uv,\ee
whose associated quadratic form is precisely
$$Q(u)=\int_\Omega|\nabla u|^2+\int_{\partial\Omega}W''(\phi)u^2,\quad u\in H.$$
It is easy to check that $B$ is a continuous and symmetric bilinear functional in $H$.\\

The key lemma in this section is the following:

\begin{lem}\label{lem:1}{\bf (Lower bound for $Q$).}
There is some constant $C_0>0$ such that for all $u\in H_0$, where $H_0$ is defined in \eqref{space}, we have that
$$Q(u)\geq C_0\int_{\partial\Omega} u^2.$$
\end{lem}

So we immediately have that
\begin{cor}\label{cor-coercivity} {\bf (Coercivity).}
There exists some $C>0$ such that for all $u\in H_0$,
$$B(u,u)\geq C\norm{u}_H^2.$$
\end{cor}

\begin{proof}
It easily follows from Lemma \ref{lem:1}. Indeed,
\begin{equation*}
\begin{split}
B(u,u) & =\int_\Omega \abs{\nabla u}^2+\int_{\partial\Omega}W''(\phi)u^2\geq (1-\mu) C_0 \int_{\partial\Omega} u^2+\mu Q(u)\\
& \geq \mu\int_\Omega \abs{\nabla u}^2+\int_{\partial\Omega} u^2\left[(1-\mu)C_0+\mu W''(\phi)\right].\\
\end{split}\end{equation*}
Because $W''$ is bounded, there exists $\mu>0$ small enough such that $(1-\mu)C_0+\mu W''(\phi)>\tfrac{C_0}{2}$. Therefore,
\begin{equation}\label{eq::909}
B(u,u)\geq\min\left\{\mu,\tfrac{C_0}{2}\right\} \norm{u}_H^2,
\end{equation}
as claimed.
\end{proof}
\qed

\bigskip

\textbf{Proof of Lemma \ref{lem:1}:} It is divided into several steps:\\

\emph{Step 1: Stability.} Because of the construction of $u_W$ and $\phi$, we know  from Theorem \ref{thm-layer-solutions} that $Q(u)\geq 0$ for all $u\in \mathcal C_c^\infty(\bar\Omega)$ (i.e. layer solutions are stable). Thus the functional $Q$ is  bounded from below also in $H$.\\

\emph{Step 2: Existence.}\\
\emph{Step 2.1: Preliminaries.}\\
We seek a minimizer for the functional
$$Q(u)=\int_{\Omega}\abs{\nabla u}^2+\int_{\partial\Omega} W''(\phi)u^2$$
subject to the constraints
$$u\in H_0\quad\mbox{and}\quad \int_{\partial\Omega}u^2=1.$$
 Note that $W''(\phi)$ is not bounded from below by any positive constant. However, due to our assumptions $Ai)$ on $W$ and the growth condition \eqref{asymptotics} for $\phi$, we still can find some $\tau>0$ and $h$ a non-negative smooth function supported on a compact set $K:=[-R,R]\subset\mathbb R$ that makes $W''(\phi)+h\geq \tau>0$ for all $x\in\mathbb R$. We write
$$Q(u)=\left[ \int_{\Omega}\abs{\nabla u}^2+\int_{\partial\Omega} \lp W''(\phi)+h\rp u^2 \right]-\int_{K} h u^2=:Q_{1}(u)+Q_2(u).$$
Let $\{u_k\}_{k=1}^{\infty}\subset H_0$ be a minimizing sequence for $Q$ with $\int_{\partial\Omega} u_k^2=1$. This $L^2$ bound for $u_k$ tells us that, in particular, $\abs{Q_2(u_k)}\leq \sup h\leq C$ for all $k$, for some constant $C>0$. As a consequence, we also know that there exists a constant $C'$ such that $Q_1(u_k)\leq C'$. From here we get a uniform bound for the $H$-norm of the sequence $\{u_k\}$, depending on $\tau$ and $C'$.\\

Next, we concentrate on the space $H^\half(\mathbb R)$. We define an equivalent Hilbert space norm $\norm{\cdot}_{\tilde H}$ on $\mathbb R$ as follows: for any $w\in H^\half(\mathbb R)$, we consider its harmonic extension $u:=P*_x w$, and define
$$\norm{w}_{\tilde H}:=Q_1(u).$$

We have shown that $\norm{u_k}_H\leq C$ for all $k$. This implies that for $w_k:=T u_k$, the norm $\norm{w_k}_{\tilde H}\leq C'$ for another constant $C'$. Because $\tilde H$ is a Hilbert space, then there is a subsequence (still denoted by $w_k$), such that $w_k \rightharpoondown w_0$, weakly in $\tilde H$, for some $w_0\in H^\half(\mathbb R)$. Weak convergence implies, for instance (\cite{Brezis:book}, proposition III.12), that
$$\norm{w_0}_{\tilde H}\leq \liminf_{k\to\infty}\norm {w_k}_{\tilde H}.$$
Set $\tilde u_k:=P *_x w_k$, that has the same trace as $u_k$, and $u_0:=P*_x w_0$. It is clear that
\be\label{Q_1}Q_1(\tilde u_k)\leq Q_1(u_k)\ee
and that equality holds when $u_k$ is already harmonic.
Then we have shown that
\be\label{Q1} Q_1(u_0)=\norm{w_0}_{\tilde H}\leq \liminf_{k\to\infty} \norm{w_k}_{\tilde H}=Q_1(\tilde u_k)\leq Q_1(u_k).\ee

On the other hand, the embedding $H^\half(\R) \hookrightarrow L^p(K)$ is compact for any compact subset $K\subset \mathbb R$ and $p>2$, as stated in Theorem \ref{Sobolev}. Thus, we can find some subsequence (still denoted by $w_k$) such that $w_k\to \bar w_0$ strongly in $L^p(K)$. As a consequence, also $w_k\to \bar w_0$ in $L^2(K)$. The uniqueness of the limit will give that $\bar w_0=w_0$, at least a.e. in $K$. In addition, since $Q_2$ lives only on the compact subset $K\subset \R$, we also have that
\be\label{Q2} Q_2(u_k)\to Q_2(u_0).\ee

From \eqref{Q1} and \eqref{Q2} we deduce that $u_0\in H$ satisfies
$$Q(u_0)\leq \liminf_{k\to \infty} Q(u_k).$$

Remark that if $\liminf_{k\to\infty} Q(u_k) >0$, then we have already finished the proof, because this implies (\ref{eq::909}).
Therefore we can assume that
\begin{equation}\label{eq::910}
\liminf Q(u_k) =0
\end{equation}

\emph{Step 2.2: Concluding.}\\
\emph{Case A: $u_0\equiv 0$}\\
Remark first that
$$Q_1(u_k)\ge \tau \int_{\partial \Omega} u_k^2 = \tau >0.$$
On the other hand, we have $Q_2(u_k)\to Q_2(u_0)=0$. Therefore we get
$$\liminf_{k\to\infty} Q(u_k) = \liminf_{k\to\infty} \left\{Q_1(u_k)+ Q_2(u_k)\right\} \ge \tau >0,$$
which is in contradiction with (\ref{eq::910}). Therefore only the next case occurs.\\
\emph{Case B: $u_0\not\equiv 0$}\\
Then we have $0\le Q_1(u_0)=0$ and we will show that $Tu_k\to T u_0$ in $H^\half(\mathbb R)$.
Since $\{u_k\}$ was a minimizing sequence, then $Q(u_k)\to 0= Q(u_0)$ when $k\to\infty$.
In addition, we have shown that $Tu_k\to Tu_0$ strongly in $L^2(K)$, so that $Q_2(u_k)\to Q_2(u_0)$.
Because $Q(u_0)=0\le Q(\tilde{u}_k)\le Q(u_k) \to 0$, we deduce that $Q(\tilde{u}_k)\to Q(u_0)$.
Moreover we have  $Q_2(\tilde{u}_k)= Q_2(u_k) \to Q_2(u_0)$ and $Q=Q_1+Q_2$ which implies that
\be\label{limit1}Q_1(\tilde u_k) \to \ Q_1(u_0).\ee
By definition \eqref{limit1} means that
\be\label{norms-are-equal}
\lim_{k\to\infty} \norm{w_k}_{\tilde H}=\norm{w_0}_{\tilde H}.\ee
Proposition III.30 in \cite{Brezis:book} assures that if we have a weakly convergent sequence $w_k\rightharpoondown w_0$ in the space $\tilde H$ and such that we have convergence of the norms \eqref{norms-are-equal}, then the convergence  $w_k\to w_0$ is strong in $\tilde H$, and then in $H^{\half}(\mathbb R)$.

As a consequence, $w_k=Tu_k\to Tu_0=w_0$ in $L^2(\mathbb R)$, which, in particular, implies the non-degeneracy of the $L^2(\partial\Omega)$ norm of $w_0$:
$$\int_{\partial\Omega}w_0^2=\lim_{k\to\infty}\int_{\partial\Omega} w_k^2 =1.$$

To finish, just remark that strong convergence in $L^2(\mathbb R)$ implies that $\int_{\partial\Omega} u_0 \phi'=0$. Thus, $u_0\in H_0$.\\

\emph{Step 3: Euler-Lagrange equation.} Once the minimizer $u_0$ is found, the theory of Lagrange multipliers for minimization problems (see, for instance, \cite{Evans}, chapter 8.4) gives that $u_0\in H$ must be a weak solution of
\be\label{Lagrange}B(u_0,v)=\lambda \int_{\partial\Omega} u_0v+\mu \int_{\partial\Omega} \phi'v \quad \mbox{ for every } v\in H,\ee
where $\lambda,\mu\in\mathbb R$. The value of the multipliers can be found by choosing the right test function $v$. Thus if we test the equality \eqref{Lagrange} with $v=P*_x\phi'$, we obtain $\mu=0$ (see Lemma \ref{properties}), while testing against $v=u_0$ gives that
\be\label{eq::100}\lambda=Q(u_0)=\int_{\Omega}\abs{\nabla u_0}^2+\int_{\partial\Omega} W''(\phi)u_0^2 \quad (\geq 0).\ee

\emph{Step 4: Regularity.} First, Proposition \ref{prop-moser-iteration} with $d(x)=W''(\phi(x))-\lambda$ implies a uniform $L^\infty$ bound for both $w_0$ and $u_0=P*_x w_0$. Then, Proposition \ref{prop-regularity} assures that $u_0\in\mathcal C^{1,\beta}_{\mbox{loc}}(\bar\Omega)$, with the following boundary estimate
$$\norm{u_0}_{\mathcal C^{1,\beta}(\overline{B_R^+})}\leq C\lp \beta,R,\norm{u_0}_{L^\infty(B_{4R}^+)}, \norm{W''(\phi)}_{\mathcal C^{\beta}([-4R,4R])}\rp.$$
Note that $u_0$ is as smooth as we want in the interior.

\emph{Step 5: Positivity.} In the following we show that if $\lambda=0$, then  $w_0=Tu_0$ must necessarily be a multiple of $\phi'$, which is a contradiction with the fact that $u_0\in H_0$ and $\int_{\partial \Omega} u_0^2=1$.

Thus we take $\lambda=0$ in \eqref{Lagrange}. Then we have a solution $u_0$ satisfying $B(u_0,v)=0$ for all $v\in H$. We remind the reader that have defined $u_W$ to be the layer solution for our potential $W$, with trace $\phi:=u_W(\cdot,0)$, and $\bar u:=\partial_x u_W$, which is also uniformly bounded and has trace $\bar u(\cdot, 0)=\phi'$.
Because of the hypothesis on the potential $W$ and the asymptotic behavior of the layer solution $\phi$, there exists some $\tau>0$ such that  $W''(\phi)\geq \tau$ on $\mathbb R\backslash[-R,R]$ for some $R>0$.

We have shown, in particular, that $u_0$ is uniformly bounded in $\overline{\Omega}$. Since $\bar u$ is positive from \eqref{derivative-positive}, then there exists $\kappa\geq 0$ such that $\abs{u_0}<\kappa \bar u$ in $B^+_R$. We set $v:=\kappa \bar u-u_0$, and note that $v$ satisfies the conditions of Proposition \ref{prop-weak-comparison}. Thus we can conclude that $v\geq 0$ in $\overline{\mathbb R^2_+}$ so that $u_0\leq\kappa\bar u$. An analogous argument gives that $-u_0<\kappa\bar u$. We have proved that $\abs{u_0}<\kappa\bar u$ on $\overline{\mathbb R^2_+}$ and, in particular, $\abs{w(x)}<\kappa\phi'(x)$ for all $x\in\mathbb R$.\\

The next step is to decrease $\kappa$ up to the first point where $\kappa\bar u$ touches $\abs{u_0}$: assume, without loss of generality, that for some $\kappa^*\geq 0$ there exists $(x_0,y_0)\in\overline{\mathbb R^2_+}$ such that
\be\label{functions-touch}\kappa^* \bar u=u_0 \quad \mbox{at }(x_0,y_0)\quad \mbox{and}\quad \kappa^*\bar u \ge u_0 \quad \mbox{on}\quad \overline{\Omega}.\ee
We set $d_1:=\max\{W''(\phi),\tau\}$, and $d_2:=d_1-W''(\phi)$. We note that both are non-negative functions on $\mathbb R$ and  $d_1(x)\geq \tau$ for all $x\in\mathbb R$. Let $v:=\kappa^*\bar u-u_0\ge 0$, that is a weak supersolution of
\begin{equation*}\left\{\begin{split}
&\Delta v = 0 \quad &\mbox{ in } \Omega,\\
&\frac{\partial  v}{\partial \nu} +d_1 v=d_2 v \geq 0\quad &\mbox{ on }\partial\Omega.
\end{split}\right.\end{equation*}
We are in the hypothesis of Lemma \ref{lema-comparison}, that can be applied because the functions have enough regularity (in particular $v\in\mathcal C^{1,\beta}_{\mbox{loc}}(\overline\Omega)$). However, because of \eqref{functions-touch} we necessarily must have $v\equiv 0$. We have shown that $u_0\equiv\kappa^*\bar u$, which at the boundary translates as $w\equiv\kappa^*\phi'$.\\

\emph{Step 6: Conclusion.}
The multiplier $\lambda$ in \eqref{eq::100} must be non-negative. However, we have shown that $\lambda\neq 0$, so then
\be\label{eq:100}
Q(u)\geq Q(u_0)=\lambda=: C_0>0 \quad\mbox{for all }u\in H_0 \mbox{ with }\int_{\partial\Omega} u^2=1.\ee
Finally, substituting
$$u_1:=\frac{u}{\sqrt{\int_{\partial\Omega} u^2}}$$
for any $u\in H_0$ into \eqref{eq:100}
finishes the proof of Lemma \ref{lem:1}.

\qed

\subsection{Proof of Theorem \ref{th::6} - regularity}

Here we seek a solution for
\begin{equation}\label{eq:20}
L\psi -W''(\phi)\psi=  g
\end{equation}
where
$$g:=\phi'+\eta(W''(\phi)-W''(0)) \quad \mbox{with}\quad \eta=\frac{\int (\phi')^2}{W''(0)}.$$
It is easy to check that
$$g\in H^\half(\mathbb R),\quad\int_{\partial\Omega} g\phi'=0.$$

We define also the continuous linear functional on $H$ given by
$$Ev:=-\int_{\partial\Omega} gTv.$$
The previous Corollary \ref{cor-coercivity} asserts that the bilinear functional $B$ from \eqref{bilinear-functional} is coercive on the space $H_0$. Then, Lax-Milgram theorem can be applied to obtain a unique weak solution $u_0\in H$ such that $B(u_0,v)=Ev$ for all $v\in H$.
Moreover, the solution depends continuously on the initial datum:
$$\norm{u_0}_H\leq C \norm{g}_{H^\half(\mathbb R)}.$$
We set $\psi:=Tu_0\in H^\half(\mathbb R)$.\\

The following theorem summarizes its regularity and properties:

\begin{theo}  The solution $\psi\in H^\half(\mathbb R)$ of (\ref{eq:20}) has regularity $\mathcal C^{1,\beta}_{\mbox{loc}}(\mathbb R)\cap L^\infty (\R)$,  satisfies $\psi(\pm\infty)=0$ and $\norm{\psi'}_{L^{\infty}(\mathbb R)}<\infty$.
\end{theo}

\begin{proof}
Note that $u_0\in H$ so, in particular, $\psi\in H^\half(\mathbb R)$ (and is also in all $L^p(\mathbb R)$ for $p>1$).\\

\emph{Claim 1:} $\psi$ is bounded; as a consequence, also $u_0$. This follows from Corollary \ref{cor-Moser}, after estimating the norm of $g$. We use the asymptotic behavior and boundedness of $\phi,\phi'$ from Theorem \ref{th::5}, and the H\"{o}lder regularity of $W''$. Then
$$\int_{\mathbb R} \abs{g}^{p}\leq \int_{(-R,R)} c^{p}+\int_{\mathbb R\backslash (-R,R)} \frac{c^{p}}{\abs{x}^{p\beta}}\leq c^{p}\left[ 2R+\frac{2}{(p\beta-1)R^{p\beta-1}}\right].$$
In particular, for $R=1$, and $p>\frac{1}{\beta}$ we have,
$$\norm{g}_{L^{p}(\mathbb R)}\leq K c.$$

\emph{Claim 2: } $\mathcal C^{1,\beta}_{\mbox{loc}}(\mathbb R)$ regularity for $\psi$ follows from Proposition \ref{prop-regularity},
while the uniform bound for $\psi'$ is a consequence of Corollary \ref{cor-regularity}. For the values of $\psi$ at infinity, just note that $\psi$ is a smooth function in $L^2(\mathbb R)$.\end{proof}
\qed

Finally, a weak solution $\psi\in H^\half(\mathbb R)\cap \mathcal C^{1,\beta}_{\mbox{loc}}(\mathbb R)\cap L^\infty (\R)$ of \eqref{eq:20} is actually a viscosity  solution of the equation \eqref{eq::7}, thus completing the proof of Theorem \ref{th::6}. This is so because of Lemma \ref{lemma-equivalence}.\\



\noindent {\bf Acknowledgements}\\
This work was supported by the ACI ``Dislocations'' (2003-2007), and by Spanish projects MTM2008-06349-C03-01 and GenCat 2009SGR345.





\begin{thebibliography}{10}

\bibitem{Alberti-Bouchitte-Seppecher:singular-perturbations}
G.~Alberti, G.~Bouchitt{\'e}, and P.~Seppecher.
\newblock Un r\'esultat de perturbations singuli\`eres avec la norme {$H\sp
  {1/2}$}.
\newblock {\em C. R. Acad. Sci. Paris S\'er. I Math.}, 319(4):333--338, 1994.

\bibitem{Alberti-Bouchitte-Seppecher:Phase-transition}
G.~Alberti, G.~Bouchitt{\'e}, and P.~Seppecher.
\newblock Phase transition with the line-tension effect.
\newblock {\em Arch. Rational Mech. Anal.}, 144(1):1--46, 1998.

\bibitem{Alvarez-Hoch-LeBouar-Monneau:dislocation-dynamics}
O.~Alvarez, P.~Hoch, Y.~Le~Bouar, and R.~Monneau.
\newblock Dislocation dynamics: short-time existence and uniqueness of the
  solution.
\newblock {\em Arch. Ration. Mech. Anal.}, 181(3):449--504, 2006.


\bibitem{Barles-Imbert:viscosity-solutions}
G.~Barles and C.~Imbert.
\newblock Second-order elliptic integro-differential equations: viscosity
  solutions' theory revisited.
\newblock {\em Ann. Inst. H. Poincar\'e Anal. Non Lin\'eaire}, 25(3):567--585,
  2008.


\bibitem{Brezis:book}
H.~Brezis.
\newblock {\em Analyse fonctionnelle}.
\newblock Collection Math\'ematiques Appliqu\'ees pour la Maitrise. [Collection
  of Applied Mathematics for the Master's Degree]. Masson, Paris, 1983.
\newblock Th\'eorie et applications. [Theory and applications].

\bibitem{Bronsard-Kohn:slowness}
L.~Bronsard and R.~V. Kohn.
\newblock On the slowness of phase boundary motion in one space dimension.
\newblock {\em Comm. Pure Appl. Math.}, 43(8):983--997, 1990.

\bibitem{Cabre-Sire:layer-solutions}
X.~Cabr{\'e} and Y.~Sire.
\newblock Non-linear equations for fractional {L}aplacians {I}: regularity,
  maximum principles and {H}amiltoniam estimates.
\newblock In preparation.

\bibitem{Cabre-SolaMorales:layer-solutions}
X.~Cabr{\'e} and J.~Sol{\`a}-Morales.
\newblock Layer solutions in a half-space for boundary reactions.
\newblock {\em Comm. Pure Appl. Math.}, 58(12):1678--1732, 2005.

\bibitem{Caffarelli-Silvestre}
L.~Caffarelli and L.~Silvestre.
\newblock An extension problem related to the fractional {L}aplacian.
\newblock {\em Comm. Partial Differential Equations}, 32(7-9):1245--1260, 2007.

\bibitem{Carr-Pego:metastable-patterns}
J.~Carr and R.~L. Pego.
\newblock Metastable patterns in solutions of {$u\sb t=\epsilon\sp 2u\sb
  {xx}-f(u)$}.
\newblock {\em Comm. Pure Appl. Math.}, 42(5):523--576, 1989.

\bibitem{Chen}
X. Chen.
\newblock Generation, propagation, and annihilation of metastable patterns.
\newblock J. Differential Equ. 206, 399-437, 2004.


\bibitem{DaLio-Forcadel-Monneau:convergence}
F.~Da~Lio, N.~Forcadel, and R.~Monneau.
\newblock Convergence of a non-local eikonal equation to anisotropic mean
  curvature motion. {A}pplication to dislocation dynamics.
\newblock {\em J. Eur. Math. Soc. (JEMS)}, 10(4):1061--1104, 2008.

\bibitem{Droniou-Imbert:fractal-first-order}
J.~Droniou and C.~Imbert.
\newblock Fractal first-order partial differential equations.
\newblock {\em Arch. Ration. Mech. Anal.}, 182(2):299--331, 2006.

\bibitem{Denoual}
C. Denoual.
\newblock Dynamic dislocation modeling by combining Peierls Nabarro and Galerkin methods
\newblock Phys. Rev. B 70, 024106, 2004.




\bibitem{EHIM}
A.~El Hajj, H.~Ibrahim and R.~Monneau.
\newblock Dislocation dynamics: from microscopic models to macroscopic crystal plasticity.
\newblock {\em Continuum Mechanics and Thermodynamics} 21(2):109-123, 2009.


\bibitem{Ei}
S.-I. Ei.
\newblock The Motion of Weakly Interacting Pulses in Reaction-Diffusion Systems.
\newblock {\em J. Dynamics Differential Equ.} 14(1),  85-137, 2002.


\bibitem{Evans}
L.~C. Evans.
\newblock {\em Partial differential equations}, volume~19 of {\em Graduate
  Studies in Mathematics}.
\newblock American Mathematical Society, Providence, RI, 1998.


\bibitem{Fabes-Kenig-Serapioni:local-regularity-degenerate}
E.~B. Fabes, C.~E. Kenig, and R.~P. Serapioni.
\newblock The local regularity of solutions of degenerate elliptic equations.
\newblock {\em Comm. Partial Differential Equations}, 7(1):77--116, 1982.

\bibitem{FinoIM}
A.~Fino, H.~Ibrahim, R.~Monneau.
\newblock Work in progress.


\bibitem{Forcadel-Imbert-Monneau:homogeneization}
N.~Forcadel, C.~Imbert, and R.~Monneau.
\newblock Homogenization of some particle systems with two-body interactions
  and of the dislocation dynamics.
\newblock {\em Discrete Contin. Dyn. Syst.}, 23(3):785--826, 2009.

\bibitem{Fusco-Hale:slow-motion}
G.~Fusco and J.~K. Hale.
\newblock Slow-motion manifolds, dormant instability, and singular
  perturbations.
\newblock {\em J. Dynam. Differential Equations}, 1(1):75--94, 1989.

\bibitem{Garroni-Muller:Gamma-limit-dislocations}
A.~Garroni and S.~M{\"u}ller.
\newblock {$\Gamma$}-limit of a phase-field model of dislocations.
\newblock {\em SIAM J. Math. Anal.}, 36(6):1943--1964 (electronic), 2005.

\bibitem{Gilbarg-Trudinger}
D.~Gilbarg and N.~S. Trudinger.
\newblock {\em Elliptic partial differential equations of second order}, volume
  224 of {\em Grundlehren der Mathematischen Wissenschaften [Fundamental
  Principles of Mathematical Sciences]}.
\newblock Springer-Verlag, Berlin, second edition, 1983.

\bibitem{Grant}
C.P.~Grant.
\newblock Slow motion in one-dimensional Cahn-Morral systems.
\newblock SIAM J. Math. Anal. 26, 21-34, 1995.


\bibitem{HL}
J.R.~Hirth and L.~Lothe.
\newblock {\em Theory of dislocations}.
\newblock Second Edition. Malabar, Florida: Krieger, 1992.

\bibitem{IS}
C.~Imbert, P.E.~Souganidis.
\newblock Phasefield theory for fractional diffusion-reaction equations and applications.
\newblock preprint HAL: hal-00408680.


\bibitem{KVW}
W.D.~Kalies, R.C.A.M. Van der Vorst, T. Wanner.
\newblock Slow motion in higher-order systems and $\Gamma$-convergence in one space dimension.
\newblock Nonlinear Anal. Ser. A: Theory Methods, 44(1), 33-57, 2001.


\bibitem{Levy}
P.~L{\'e}vy.
\newblock Sur les int\'egrales dont les \'el\'ements sont des variables
  al\'eatoires ind\'ependantes.
\newblock {\em Ann. Scuola Norm. Sup. Pisa Cl. Sci. (2)}, 3(3-4):337--366,
  1934.

\bibitem{Lieb-Loss}
E.~H. Lieb and M.~Loss.
\newblock {\em Analysis}, volume~14 of {\em Graduate Studies in Mathematics}.
\newblock American Mathematical Society, Providence, RI, second edition, 2001.

\bibitem{Mazja}
V.~G. Maz'ja.
\newblock {\em Sobolev spaces}.
\newblock Springer Series in Soviet Mathematics. Springer-Verlag, Berlin, 1985.
\newblock Translated from the Russian by T. O. Shaposhnikova.

\bibitem{MP}
R.~Monneau and S.~Patrizi.
\newblock Work in progress.

\bibitem{MBW}
A.B.~Movchan, R.~Bullough, J.R.~Willis.
\newblock  Stability of a dislocation: discrete model.
\newblock  Eur. J. Appl. Math. 9, 373-396, 1998.

\bibitem{N}
F.R.N. Nabarro.
\newblock Fifty-year study of the Peierls-Nabarro stress.
\newblock Material Science and Engineering A 234-236, p. 67-76, 1997.


\bibitem{Silvestre:thesis}
L.~Silvestre.
\newblock Regularity of the obstacle problem for a fractional power of the
  {L}aplace operator.
\newblock PhD thesis, 2005.

\bibitem{Stein-Weiss:Fourier-analysis}
E.~M. Stein and G.~Weiss.
\newblock {\em Introduction to {F}ourier analysis on {E}uclidean spaces}.
\newblock Princeton University Press, Princeton, N.J., 1971.
\newblock Princeton Mathematical Series, No. 32.

\bibitem{Toland:Peierls-Nabarro}
J.~F. Toland.
\newblock The {P}eierls-{N}abarro and {B}enjamin-{O}no equations.
\newblock {\em J. Funct. Anal.}, 145(1):136--150, 1997.

\bibitem{WXM}
H. Wei, Y. Xiang, P. Ming.
\newblock A Generalized Peierls-Nabarro Model for Curved Dislocations Using Discrete Fourier Transform.
\newblock Communications in computational physics 4(2), 275-293, 2008.




\end{thebibliography}
\end{document}